\documentclass{article}
\usepackage[lang=american]{ems-jems} %% change to `american' if you use American English

\usepackage[capitalise]{cleveref}

\usepackage{amscd}
\usepackage{amsfonts}
\usepackage{mathtools}

\newtheorem*{thm'}{Theorem 1.8'}

\newtheorem{thm}{Theorem}[section]
\newtheorem{lm}[thm]{Lemma}
\newtheorem{defn}[thm]{Definition}
\newtheorem{prop}[thm]{Proposition}
\newtheorem{coro}[thm]{Corollary}
\newtheorem{rmk}[thm]{Remark}

\newtheorem{example}[thm]{Example}

\newcommand{\norm}[1]{\lVert#1\rVert}
\newcommand{\abs}[1]{\left\vert#1\right\vert}
\newcommand{\op}[1]{\left(#1\right)}
\newcommand{\eps}{\varepsilon}
\newcommand{\defeq}{\vcentcolon=}

\newcommand{\one}{\mathbf{1}}
\newcommand{\R}{\mathbf{R}}
\newcommand{\C}{\mathbf{C}}
\newcommand{\h}{\mathbf{H}}
\newcommand{\Z}{\mathbf{Z}}
\newcommand{\N}{\mathbf{N}}

\newcommand{\SL}{\mathrm{SL}}
\newcommand{\GL}{\mathrm{GL}}
\newcommand{\PSL}{\mathrm{PSL}}
\newcommand{\G}{\Gamma}
\newcommand{\g}{\gamma}

\newcommand{\bx}{{\bf x}}
\newcommand{\by}{{\bf y}}
\newcommand{\bo}{{\bf 0}}
\newcommand{\be}{{\bf e}}

\newcommand{\cN}{\mathcal{N}}
\newcommand{\cU}{\mathcal{U}}

\newcommand{\gs}{\sigma}

\newcommand{\fa}{\mathfrak{a}}
\newcommand{\fb}{\mathfrak{b}}
\newcommand{\fab}{\mathfrak{ab}}

\newcommand{\bk}{\backslash}
\newcommand{\re}{\mathfrak{Re}}

\newcommand{\bbm}{\begin{bmatrix}}
\newcommand{\ebm}{\end{bmatrix}}
\newcommand{\bpm}{\begin{pmatrix}}
\newcommand{\epm}{\end{pmatrix}}
\newcommand{\bsm}{\left(\begin{smallmatrix}}
\newcommand{\esm}{\end{smallmatrix}\right)}

\numberwithin{equation}{section}

\begin{document}

%------
% Insert the title of your paper and (if necessary)
% a short title for the running head.
%------
\title{Pairs in discrete lattice orbits\\ with applications to Veech surfaces}
\titlemark{Pairs in discrete lattice orbits}

%------
% Insert full names of the authors.
% Add further authors as follows:
%  \emsauthor{2}{}{}
%  \emsauthor{3}{}{}
% etc.
% Abbreviate first names for the running head.
%------
\emsauthor{1}{Claire Burrin}{C.~Burrin}
\emsauthor{2}{Samantha Fairchild}{S.~Fairchild}
\emsauthor{3}{with an appendix by Jon Chaika}{J.~Chaika}
%------
% Use \authormark if the list of authors is too
% long for the running head: \authormark{A.~Doe et al.}
%------

%------
% Add one \emsaffil and one \email for each author.
%------
\emsaffil{1}{University of Zurich Institute of Mathematics,
Winterthurerstrasse 190,
8057 Z\"{u}rich Switzerland;\email{claire.burrin@math.uzh.ch}}
\emsaffil{2}{Max Planck Institute for Mathematics in the Sciences, Inselstraße 22, 04103 Leipzig Germany; \email{samantha.fairchild@mis.mpg.de}}
\emsaffil{3}{University of Utah Department of Mathematics 203, 155 S 1400 E RM 233, Salt Lake City, UT, 84112-0090, USA; \email{ chaika@math.utah.edu }}

%------
% Add MSC 2020 codes according to www.ams.org/msc/msc2020.html.
% Secondary codes (in square brackets) are optional.
%------

%------
% Add a list of keywords.
%------

%------
% Insert your abstract.
%------
\begin{abstract}
Let $\Lambda_1$, $\Lambda_2$ be two discrete orbits under the linear action of a  lattice $\G<\SL_2(\R)$ on the Euclidean plane. We prove a Siegel--Veech-type integral formula for the averages
\[
\sum_{\bx\in\Lambda_1} \sum_{\by\in\Lambda_2} f(\bx, \by)
\]
from which we derive new results for the set $S_M$ of holonomy vectors of saddle connections of a Veech surface $M$. This includes an effective count for generic Borel sets with respect to linear transformations, and upper bounds on the number of pairs in $S_M$ with bounded determinant and on the number of pairs in $S_M$ with bounded distance. This last estimate is used in the appendix to prove that for almost every $(\theta,\psi)\in S^1\times S^1$ the translation flows $F_\theta^t$ and $F_\psi^t$ on any Veech surface $M$ are disjoint.

\end{abstract}

\maketitle

%------
% INSERT THE BODY OF THE PAPER HERE (except
% acknowledgments, funding info and bibliography)
%------

\section{Introduction}

Let $\Gamma\subset G=\SL_2(\R)$ be a lattice acting linearly on the Euclidean plane $\R^2$. Any orbit of $\Gamma$ is either a dense or discrete subset of $\R^2$. When the orbit is dense, a limiting distribution was computed by Ledrappier \cite{Led} and extended to more general lattices in locally compact groups by Gorodnik and Weiss \cite{GW}. When the orbit is discrete, the lattice must be nonuniform and the problem of understanding the distribution is a venerable one that goes back to Gauss. 

A more recent incentive to understand the distribution of discrete lattice orbits is motivated by the study of rational polygonal billiards and of the linear flow $F_\theta^t$ (with direction $\theta \in S^1$) on translation surfaces. 
%Given a rational polygonal billiard table, the associated translation surface $M$ is constructed by taking all of the reflections of the polygon to obtain a rational polygon with opposite sides identified by Euclidean translation. The translation surface $M$ is a compact Riemann surface with a flat metric away from saddle points that arise from the polygon's vertices. The billiard flow corresponds to the translation flow $F_\theta^t$ on $M$, which is determined by a prescribed direction $\theta\in S^1$. 
A complicating aspect in the study of the translation flow is the presence of the saddle points --- for the billiard flow, the future trajectory of a ball hitting a corner of the polygonal table is ill-defined. Corner-to-corner trajectories correspond to finite geodesics on the translation surface called saddle connections. The set $S_M$ of holonomy vectors of saddle connections on a translation surface $M$ is a discrete planar set in $\R^2$ that records the length and direction of each saddle connection. For a typical translation surface, the number $|S_M\cap B_R|$ of saddle connections of length $\|\bx\|<R$ grows quadratically, with a growth asymptotic of the form
\begin{align}\label{countM}
|S_M\cap B_R| = c_M R^2 + O\left(R^{2-\delta}\right)
\end{align}
\cite{EM,NRW}, where $\delta>0$ is a nontrivial but nonexplicit power saving. 

%The leading constant $c_M$ is given by the Siegel--Veech mean value theorem \cite{Veech98} in the same way the volume constant for the space of unimodular lattices is extracted from Siegel's classical mean value theorem in the geometry of numbers. Finer aspects of the distribution of $S_M\subset\R^2$ have also been explored. For instance, Masur \cite{Masur86} showed that the set of directions in $S_M$ is dense, and  Athreya and Chaika \cite{AC} studied its gap distribution. %Vorobets \cite{Vor} showed that for a typical translation surface, the set of directions in fact becomes equidistributed. %In another direction, Chaika and Robertson \cite{CR} recently showed that on a typical translation surface, the lengths of saddle connections are equidistributed mod 1.

In the case of Veech surfaces\footnote{ We recall that a translation surface is a Veech surface if the image in $\PSL_2(\R)$ of its stabilizer group under the action of $\SL_2(\R)$ on the moduli space of all translation surfaces is a lattice. Veech surfaces are sometimes called lattice surfaces.} the set $S_M$ of holonomy vectors of saddle connections is a finite disjoint union of discrete lattice orbits \cite{Veech89}, which may then be accessed using a combination of tools and ideas from dynamics, ergodic theory, spectral theory, or metric geometry. Hence one can hope to say more, and in particular to specify results to individual surfaces. This is of interest considering for instance that the set of surfaces constructed from rational billiard polygons have null measure in the stratum. As such, results for typical surfaces do not yield new information on billiards in rational polygons.

% In the case of Veech surfaces, one can hope to say more, and in particular to specify results to individual surfaces. This is of interest considering for instance that the set of surfaces constructed from rational billiard polygons have null measure in the stratum, and as such results for typical surfaces do not yield new information on billiards in rational polygons. 
 %Veech surfaces are rare but rich examples of translation surfaces, notably for exhibiting striking dynamical properties. For example, given any direction $\theta$, the flow $F_\theta^t$ is either uniquely ergodic or completely periodic \cite{Veech89} --- a classical phenomenon in the case of flat tori. 
 %For our discussion, the relevant feature of a Veech surface $M$ is that the discrete set $S_M$ of holonomy vectors of saddle connections is a finite disjoint union of discrete lattice orbits (see \cite{Veech89})% and \cref{sec:discrete} of this paper)
 %, which may then be accessed using a combination of tools and ideas from dynamics, ergodic theory, spectral theory, metric or algebraic geometry.
% Some results on the value of the constants are known EO EMZ 

Our main theorem is an explicit mean value formula for the Siegel--Veech transform of pairs of discrete lattice orbits in the plane. %Along the way, we extend previous work \cite{Fai21} of the second author on higher moments for the Siegel--Veech transform and we answer a question of Athreya, Cheung, and Masur \cite[Section 4.5.3]{ACM} about the description of the arising Siegel--Veech measures. 
We postpone the statement of the main theorem to \cref{sec:MVT} to first describe some applications to the study of holonomies for Veech surfaces. The main novelty of this paper is to obtain estimates towards counting pairs of vectors in $S_M$. For this we build on tools from the geometry of numbers, the spectral theory of automorphic forms, the weak mixing of the translation flow for nonarithmetic Veech surfaces established by Avila and Delecroix \cite{AD}, and previous works of the authors \cite{BNRW,CH,Fai21}.

\subsection{Applications}

Let $M$ be a Veech surface, $\G_M$ its Veech group, and $S_M$ the set of holonomy vectors of its saddle connections. The results we present here for $S_M$ are proven in the main text in the more general setting of discrete lattice orbits in the plane. References to those more general statements are indicated in \cref{sec:disc}.

\subsubsection{Counting in Borel sets} 
%Using Eisenstein series, a strategy of Veech \cite{Veech89} and a remark of Sarnak \cite[Rmk.~1.12]{Veech92}, the first author with Nevo, R\"uhr, and Weiss \cite{BNRW} showed that there exists a constant $c_0>0$ such that
In the case of Veech surfaces, we have 
\begin{align}\label{asymp}
|S_M\cap B_R| = c_M R^2 + O\left(R^{2-\delta}\right),
\end{align}
\cite{Veech89,BNRW} where  
%$\delta=\tfrac23$ if $\G$ has trivial residual spectrum, and $\delta\in(0,\tfrac23)$ otherwise.\footnote{ The upper bound of $2/3$ is an artefact of the method of proof. For individual lattices, a much stronger power-saving is possible; see \cref{sec:prevcount} for discrete orbits of congruence subgroups of $\SL_2(\Z)$.} In either case, 
this time the power-saving $\delta=\delta(\G_{M})$ is explicit,\footnote{  We refer the reader to \cref{sec:prevcount} for an explicit description of $\delta$, which is determined by the bottom of the residual spectrum of $\G_M$. It is worth noting that there are known constructions of lattices $\G$ for which $\delta$ can be arbitrarily close to 0 \cite{Selberg65}. 
}  in contrast to (\ref{Thm:1}). In fact the same asymptotic holds when replacing $S_M$ by $gS_M$ for any $g\in G$ --- this amounts to counting points in $S_M$ that lie in an ellipsoid centered at the origin. For more general shapes it is typical of analytic approaches that a lack of regularity of the shape's boundary leads to weaker power-savings; see \cite[Theorems 2.6, 2.7]{BNRW} or \cref{sec10} for examples. 
%Via a second moment argument, we however can recover (close to) the original power-saving for generic Borel sets in the plane.
%Using either nonspherical Eisenstein series or an ergodic argument following  Gorodnik and Nevo \cite{GN}, the disk $B_R$ may also be replaced by more general sets containing the origin; for instance, the asymptotic \eqref{asymp} also holds for $B_R$ replaced by the $R$-dilate of any star-shaped domain containing the origin with piecewise Lipschitz boundary (e.g., any convex polygon) with the weaker power-saving $\delta=\tfrac14-\epsilon$ \cite{BNRW}. However, spectral and ergodic methods require some regularity conditions on the boundary of the set and do not extend to general Lebesgue measurable sets. A second moment estimate proved in this paper allows access to counting in general Lebesgue measurable sets.
The following theorem recovers a (nearly) optimal count for typical shapes with respect to linear transformations.

\begin{thm}\label{Thm:1}
Let $\Omega\subset\R^{2}$ be a bounded Borel set that contains the origin, and consider its dilates $\Omega_R=R\cdot\Omega$.  Then for almost every\footnote{With respect to the Euclidean metric induced by the matrix representation of $A$.} linear transformation $A$ we have
\[
|A(S_M)\cap \Omega_R| |\det(A)| =c_M |\Omega| R^2+  O\left(R^{2-\delta}  \log^{3/2}(R)\right),
\]
where $\delta=\delta(\G_{M})$ is as in \eqref{asymp}.% and $\eps>0$.
\end{thm}

\begin{rmk}
The restriction to dilates is cosmetic; see \cref{countN} for a counting asymptotic that holds for every element of a linearly ordered family of Borel sets in the plane. 
\end{rmk}

\subsubsection{Weak uniform discreteness}
Recall that a discrete planar set is $\eta$-uniformly discrete if $|\Lambda\cap B_\eta(\bx)|\leq 1$ for all $\bx\in\R^2$. Answering a question of Barak Weiss, Wu showed that $S_M$ is not uniformly discrete when $M$ is a nonarithmetic Veech surface \cite{Wu}; in other words there exists a pair of vectors that are arbitrarily close. On the other hand, uniform discreteness is easy to prove when $M$ is an arithmetic Veech surface; see \cref{prop:arithuni}. The next theorem quantifies the failure of uniform discreteness when $M$ is nonarithmetic.

\begin{thm}\label{thm:epsM}
Let $M$ be a Veech surface. 
For each $\eps>0$ there is an $\eta>0$ such that
\begin{align*}%\label{eta-friends}
\limsup_{R\to\infty}\, \frac{|\{\bx\in S_M\cap B_R : |S_M\cap B_\eta(\bx)|\geq 2\}|}{|S_M\cap B_R|} <\eps.
\end{align*}
\end{thm}

The theorem shows that upon discarding an $\eps$-dense set of saddle connections, $S_{M}$ is $\eta$-uniformly discrete for some $\eta>0$.

\subsubsection{Disjointness of flows}

The translation flow $F_\theta^t$ is uniquely ergodic \cite{KMS} for Lebesgue almost every direction $\theta$. On a {\em typical} translation surface (with respect to the Lebesgue measure class on a given stratum) of genus $g\geq2$, the flow $F_\theta^t$ is weakly mixing for almost every $\theta$ \cite{AF} (but never strongly mixing \cite{Kat}). For a typical surface the flows $F_\theta^t$ and $F_\psi^t$ are disjoint (for almost every pair $(\theta,\psi)$ of directions) \cite{CH}. An immediate consequence of disjointness is that $F^t_\theta$ and $F_\psi^t$ are not isomorphic.

Regarding individual surfaces, the work of \cite{CH} can be applied to establish disjointness for branched covers of the torus in typical directions. 
%If $M$ is a branched cover of the torus, spectral arguments show that for almost every pair $(\theta,\psi)$, the translation flows $F^t_\theta$ and $F_\psi^t$ are disjoint; see \cite{CH}.  
So far, covers of tori were the only examples of individual surfaces for which disjointness is shown. %(A non-explicit construction of a full measure set of translation surface where almost every pair of translation flows is disjoint is further presented in \cite{CH}.) 
In \cref{AppJon}, the third author presents a direct proof of the following theorem based on \cref{thm:epsM}.
\begin{thm}\label{thm:disjoint}
Let $M$ be any Veech surface. Then for Lebesgue almost every pair $(\theta,\psi)\in S^1\times S^1$, the translations flows $F_\theta^t$ and $F_\psi^t$ are disjoint. 
\end{thm}
 
This provides the first family of surfaces other than branched covers of tori for which the flows in almost every pair of directions are not isomorphic. If \cref{thm:disjoint} could be extended to all surfaces, we would even be able to recover the result of Chaika and Forni that there is a weakly mixing billiard in a polygon \cite{CF}.

\subsubsection{Counting pairs with bounded determinant}

\cref{thm:epsM} follows from the following stronger density theorem that also accounts for multiplicities. Here we denote $B_\eta^*(\bx) = B_\eta(\bx)\setminus \{\bx\}$.

\begin{thm}\label{thm:ballcount}
Let $M$ be a Veech surface. There exists a constant $C$ such that%For each $\eps>0$ there is $\eta>0$ such that
\[
%\limsup_{R\to\infty} 
\frac{|\{(\bx,\by)\in S_{M}\times S_{M}: \bx\in B_{R},\, \by\in B^*_{\eta}(\bx)\}|}{|S_M\cap B_R|} < C\eta^2.
\] 
\end{thm}
The same arguments also recover an upper bound on pairs of saddle connections with bounded determinant. For a vector $\bx\in \R^2$, set 
\[\mathcal{D}_{D,1}(\bx) = \{\by \in \R^2: |\by| \leq |\bx|\text{ and } |\bx \wedge \by| \leq D\}.\] 
 For a typical surface $M$, the second author with Athreya and Masur \cite{AFM22} showed that for $D>0$ there is a {\em non-explicit} constant $C_D>0$ so that
\[\lim_{R\to\infty} \frac{|\{(\bx, \by) \in S_M \times S_M: \bx \in B_R,\,\, \by \in \mathcal{D}_{D,1}(\bx)|}{R^2} = C_D.\]

\begin{thm} \label{thm:detM}
	Let $M$ be a Veech surface. For any $D>0$, there are constants $C_M$ and $c$ depending only on $M$ so that
\[
\limsup_{R\to\infty} \frac{|\{(\bx, \by) \in S_M \times S_M: \bx \in B_R,\,\, \by \in \mathcal{D}_{D,1}(\bx)\}|}{R^2} \leq C_M (D +  c).
\]

\end{thm}
Note we have two terms in the upper bound. This comes from the fact that when $D$ is small, there are essentially only parallel pairs, and thus a constant multiple of the Siegel--Veech constant $c_\Gamma$ will be dominating (cf. \cite[Theorem~1.2]{AFM22}). However for $D$ large, we have an upper bound which is asymptotically linear in $D$.
\subsubsection{Pair correlations}

In view of these results it is natural to consider the two-dimensional pair correlation function
\[
R_2(B_s,S_M,R) := \frac{|\{(\bx,\by)\in S_M\times S_M: \bx\in B_R,\, \by\in B^*_{s/\sqrt{c_M}}(\bx)\}|}{|S_M\cap B_R|} \qquad (s>0)
\]
with $s>0$ and a rescaling determined by the mean spacing $\tfrac{|S_M\cap B_R|}{|B_R|}\sim c_M$. The pair correlation is said to be Poissonian if $R_2(B_s,S_M,R)\to |B_s|$ as $R\to\infty$. Poissonian expresses that we see the same asymptotic behavior for the pair correlation function as if we were looking at points in the plane generated by a two-dimensional Poisson process. 

Our last application shows the pair correlation function to be Poissonian on average. To make this precise we need some notation. Let $\mu$ denote the probability Haar measure on $G/\G_M$ and let $C$ be the cone 
\[
C = \{ A= \nu^{1/2}g : g\in G/\G_M,\, \det(A)=\nu\in(0,1]\}
\]
equipped with the probability measure given by
\[
\int_C f(A)\, dm(A) = \int_0^1 \int_{G/\G_M} f(\nu^{1/2}g)\, d\mu(g)\, d\nu.
\]

\begin{thm}\label{thm:pairco}
For each $s>0$ we have
\[
\lim_{R\to\infty} \int_C R_2(B_s,A(S_M),R) \det(A)\, dm(A) = |B_s|.
\]
\end{thm}
We will study the pair correlation function in follow-up work.
%dibs

\subsection{A mean value formula for pairs}\label{sec:MVT}

 Let $\G$ be a (nonuniform) lattice in $G$ that contains $-I$ with discrete orbit $\Lambda$, let $B_c(\R^2)$ denote the space of Borel measurable bounded compactly supported functions on the Euclidean plane. 
 
 In analogy to Siegel's mean value formula in the geometry of numbers, Veech \cite{Veech98} introduced the Siegel--Veech transform
\[
f\in B_c(\R^2) \mapsto \Theta_{\Lambda; f}(g) := \sum_{\bx\in \Lambda} f(g\bx)
\]
and proved that $\Theta_{\Lambda;f}\in L^\infty(G/\G)$ with first moment 
%
% in \cite[Theorems 6.4 and 6.5]{Veech98} that $\Theta_{\Lambda_M; f} \in L^1$ , satisfying
\begin{align}\label{veech}
\int_{G/\G} \Theta_{\Lambda; f} (g) d\mu(g) = c_\Lambda \int_{\R^2} f(\bx)d\bx.
\end{align}
%Recently, Athreya, Cheung, and Masur \cite{ACM} showed that for a typical translation surface $\Theta_{\Lambda_M; f}\in L^2$, while $\Theta_{\Lambda_M; f}\not\in L^3$ follows from \cite{AC}. In theory, this makes the computation of the second moment of $\Theta_{\Lambda_M; f}$ possible \cite[Appendix]{ACM}. 

The constant $c_\Lambda$ is explicit. For each discrete lattice orbit $\Lambda$ there exists a preferred scaling of $\Lambda$ in its homothety class $\{\lambda\Lambda:\lambda\neq0\}$ for which $c_{\Lambda}=c_\G:=\tfrac{2}{\pi V}$, where $V$ is the hyperbolic covolume of $\G$. We say that $\Lambda$ is a scaled discrete orbit if its constant has this form. (The precise scaling $\lambda$ can be explicitly computed; see \cref{GJ}.)

Similarly for two (not necessarily distinct) discrete $\G$-orbits $\Lambda_1$, $\Lambda_2$, and any $f\in B_c(\R^2\times\R^2)$ we consider the Siegel--Veech theta-transform
\[
\Theta_{\Lambda_1,\Lambda_2; f}:G/\Gamma\to\R,\quad \Theta_{\Lambda_1,\Lambda_2; f} (g)= \sum_{\substack{\bx\in\Lambda_1\\\by\in\Lambda_2}} f(g\bx,g\by).
\]
To simplify our formulas, we assume that $\Lambda_1$, $\Lambda_2$ are scaled so that $c_{\Lambda_i}=c_\G$. The general statement can be recovered from the observation that if $\Lambda'_i=\lambda_i\Lambda_i$, then $\Theta_{\Lambda'_1,\Lambda'_2; f}=\Theta_{\Lambda_1,\Lambda_2;f\circ\lambda}$, with $f\circ\lambda(\bx,\by):=f(\lambda_1\bx,\lambda_2\by)$.

Let 
$
\cN=\{ \det(\bx\, |\,\by) : (\bx,\by)\in\Lambda_1\times\Lambda_2\}.
$ 
Two ordered pairs $(\bx_1,\by_1)$, $(\bx_2,\by_2)\in\Lambda_1\times\Lambda_2$ are considered equivalent if $(\bx_1,\by_1)=(\g\bx_2,\g\by_2)$ for some $\g\in\G$ and we set $\varphi(c)$ to be the number of equivalence classes of pairs $(\bx,\by)\in\Lambda_1\times\Lambda_2$ with $\det(\bx|\by)=c$. 
Our main theorem is

\begin{thm}\label{Thm1}
Let $\Lambda_1$, $\Lambda_2$ be scaled discrete $\G$-orbits, and let $f$ be a bounded semicontinuous function of compact support on $\R^2\times\R^2$. Then $\Theta_{\Lambda_1,\Lambda_2;f}\in L^\infty(G/\G)$ and 
\begin{align*}
\int_{G/\G} \Theta_{\Lambda_1,\Lambda_2;f}(g)\, d\mu(g) = c_{\G} \sum_{c\in\cN-0} & \frac{\varphi(c)}{|c|} \int_\R \int_{\R^2} f(\bx, t\bx+c\bx^*)\, d\bx\, dt \\
&\qquad +\delta_{\Lambda_1,\Lambda_2} c_{\G} \int_{\R^2} \op{f(\bx,-\bx)+f(\bx,\bx)} d\bx,
\end{align*}
where $\bx^*$ is such that $\det(\bx\, |\, \bx^*)=1$ and $\delta_{\Lambda_1,\Lambda_2}=1$ if $\Lambda_1$ and $\Lambda_2$ are homothetic and is 0 otherwise.
\end{thm}

The right-hand side provides a complete description of the arising Siegel--Veech measures, answering a question of Athreya, Cheung, and Masur \cite[Sections 4.4, 4.5.3]{ACM}. 

The range of admissible test functions $f$ includes characteristic functions of open and closed measurable sets. However the evaluation of such integrals is not straightforward.  Inspired by a trick used by Schmidt in the geometry of numbers, we rewrite the mean value formula above in terms of the cone measure $m$ introduced earlier.
\begin{thm'}\label{Thm2}
With the same notation as above, we have
\begin{align*}
\int_C \Theta_{\Lambda_1,\Lambda_2; f} (A) dm(A) &= c_\G^2 \iint_{\R^2\times\R^2} f(\bx,\by)\, d\bx d\by 
+ \delta_{\Lambda_1,\Lambda_2} c_{\Gamma} \int_{\R^2} \left(f(\bx,-\bx)+f(\bx,\bx)\right)d\bx\\
&\qquad + c_\G \iint_{\R^2\times\R^2} \Psi(|\bx\wedge\by|)f(\bx,\by)\, d\bx d\by,
%c_{\Gamma} \iint_{\R^2\times\R^2} & \Phi(|\bx\wedge\by|) f(\bx,\by)d\bx d\by\\
%&\qquad + \delta_{\Lambda_1,\Lambda_2} c_{\Gamma} \int_{\R^2} \left(f(\bx,-\bx)+f(\bx,\bx)\right)d\bx,
\end{align*}
where $|\bx\wedge\by|=|\det(\bx\,|\,\by)|$, and %$\Phi(0)=0$, and  
$
\Psi(t) =  t\sum_{t\leq c\in\cN-0} \varphi(c) c^{-3} -c_\G$.
\end{thm'}

The evaluation of the last double integral therefore depends on counting determinants of pairs of linearly independent vectors in $\Lambda_1\times\Lambda_2$. Using the spectral theory of automorphic forms, we show that
\begin{thm}\label{thm:Scat}
As $t\to\infty$,
$ 
\Psi(t) = O\left(t^{-\delta}\right),
$
where $\delta$ is as in \cref{asymp}.  %is determined by the spectral gap of $\G$.
\end{thm}

Specializing to $\Lambda_1=\Lambda_2=\Lambda$ and $f(\bx,\by)=h(\bx)h(\by)$, we recover the second moment formula (where $\Lambda$ does not need to be scaled)

\begin{coro}\label{coro:2ndmoment}
\begin{align*}
\int_C \left( \sum_{\bx\in\Lambda} h(A\bx)\right)^2 dm(A) &=  \op{c_\Lambda\int_{\R^2} h(\bx)d\bx}^2+ c_\Lambda \int_{\R^2} (h(\bx)h(-\bx) +h(\bx)h(\bx)) d\bx\\
&\qquad + c_\Lambda \iint_{\R^2\times\R^2}\op{\Psi(|\bx\wedge \by|)}   h(\bx)h(\by)d\bx\, d\by.
\end{align*}
\end{coro}

% To estimate the contribution of linearly independent vectors towards applications we will rely on \cref{coro:2ndmoment} with the following estimate, which can be derived from standard results in the spectral theory of automorphic forms.
%
%\begin{thm}\label{thm:Scat}
%As $t\to\infty$
%\[ 
%\Phi(t) = c_{\Gamma} + O\left(t^{-\delta}\right),
%\]
%where $\delta=\delta(\G)$ is determined by the spectral gap of $\G$.
%\end{thm}
%

\subsection{Discussion of proofs and structure of the paper}\label{sec:disc}
Section~\ref{sec:Defintions} collects notation and initial definitions. In \cref{sec:discrete} we show that nonuniform lattices yield, up to homothety, only finitely many discrete orbits, in one-to-one correspondence to the set of inequivalent cusps for $\G$.

In Section~\ref{sec:prevcount} we show that counting problems for discrete lattice orbits can be restricted to scaled discrete orbits. We also recall the state of the art for counting results of discrete lattice orbits, and include a full proof of 
\[
N_R(\Lambda) = c_\Lambda |B_R| + o(|B_R|^{1/2})
\]
when $\Lambda$ is a discrete orbit of a congruence subgroup of $\SL_2(\Z)$; see \cref{thm:cong}. (A stronger power-saving can be achieved if we assume the Riemann hypothesis.)

In Section~\ref{sec:theta} we give a representation theorem for all higher moments of the Siegel--Veech transform over the space of semi-continuous functions.  Section~\ref{sec:Veech} reviews Veech's proof of (\ref{veech}), filling out some details left to the reader in \cite{Veech98} that help explain our restriction from measurable to semi-continuous functions when working with higher moments. 

Concluding the background part of this paper, we give in Section~\ref{sec:Admissible} a direct proof of \cref{thm:Scat} (known to follow from a more general result of Good \cite{Good}) based on the meromorphic continuation of the scattering matrix.

Our main \cref{Thm1} (mean value formulas for pairs) and its \cref{coro:2ndmoment} (second moment formulas) are proved in 
Section~\ref{sec:2ndmoment1}. The strategy of proof follows Veech's ergodic approach \cite{Veech98} and its extension by the second author to the second moment of the Siegel--Veech transform for Hecke triangle groups \cite[Theorem 1.2]{Fai21}. 
The key realization is that the geometric totient function introduced in \cite{Fai21} is a special instance of the counting function for double cosets of lattices with respect to maximal parabolic subgroups.

We are now in position to deduce the applications advertised above. \cref{sec:estimates} collects needed geometric estimates (using the spectral asymptotic provided by \cref{thm:Scat}) and concludes with the proof of \cref{thm:pairco} (Poissonian pair correlation on average).

\cref{Thm:1} (effective counting) follows directly from \cref{Thm:1DLO}, which is itself the application of the more general \cref{countN} to dilated Borel sets containing the origin. The latter theorem should be seen as the discrete lattice orbit analogue of Schmidt's bound on the discrepancy function for (primitive) lattices in $\R^2$ \cite[Theorem 2]{Schmidt}. In fact our proof mimicks Schmidt's with input the second moment formula from \cref{coro:2ndmoment} and a geometric estimate from \cref{sec:estimates}. This argument of Schmidt has recently been revisited in various contexts; see \cite{AM09}, \cite{KY21}, \cite{RSW}.

In Section~\ref{sec:friendsbounds}, \cref{thm:ballcount} and \cref{thm:detM} (counting pairs in balls or of bounded determinant) are proved using the full force of the mean value formula for pairs  and the well-roundedness of Euclidean balls. The paper concludes with the proof of Theorem~\ref{thm:disjoint} (disjointness of flows) in Appendix~\ref{AppJon}, building on a criterium for disjointness developed in \cite{CH}.

\section{Basic definitions}\label{sec:Defintions}

\subsection{Notation}\label{sec:Notation}
Given a countable set $A$, we denote by $|A|$ the number of its elements. If $A$ is a Borel set in $\R^2$, we denote its characteristic function by $\one_A$ and its Lebesgue measure by $|A|$ or $\lambda(A)$. We use $\|\cdot\|$ to denote the Euclidean norm $\|\bx\|=\sqrt{\sum x_i^2}$. For the disk  $B_R(\bx) = \{ \by\in\R^2 : \|\bx-\by\|\leq R\}$, we have $|B_R(\bx)|=\pi R^2$ and write $B_R:=B_R(\bo)$.

For two functions $f(x)$ and $g(x)$, we write $f\ll g$ or $f=O(g)$ to say that there exists a constant $C>0$ such that for $x$ sufficiently large, $|f(x)|\leq C|g(x)|$.

We set $G=\SL_2(\R)$ and distinguish the subgroups
\[
K= \left\{ k_\theta =\bpm \cos\theta&\sin\theta\\ -\sin\theta&\cos\theta\epm : \theta\in\R/2\pi\Z\right\},
\]
\[
A = \left\{ a_y = \bpm y &0\\ 0&y^{-1}\epm: y>0\right\},
\]
\[
N = \left\{ n_x = \bpm 1 & x\\ 0&1\epm: x\in\R\right\}.
\]
%The group $G$ has Bruhat decomposition $G=NA\bsm 0&-1\\1&0\esm N$ and Cartan decomposition $G=KAK$. 
The Haar measure $\eta$ on $G$ can be written in coordinates as the Lebesgue measure restricted to the set 
\begin{align}\label{mathscr-S}
\mathscr{S}=\{(a,b,s): a\neq0, b,s\in\R\}
\end{align}
with the coordinate transformation 
\[
(a,b,s)\mapsto g(a,b,s)= \bpm 1&0\\ s&1\epm \bpm a&b\\ 0 & a^{-1}\epm =\bpm a& b\\ as & bs+a^{-1}\epm.
\]

%There is a bijection between $G/\pm N$ and the set of horocycles in $\h$. There is thus a one-to-one correspondence between horocycles in $\h$ and vectors in $\mathcal{V}$. Explicitly, this correspondence is realized by the assignment
%\[
%H_t(x)=\{z\in\h: B_x(i,z)=\ln t\}  \mapsto \begin{dcases} \pm \tfrac{\sqrt{t}}{\sqrt{1+x^2}}\bpm x\\ 1\epm & \text{ if } x\neq\infty,\\
%\pm \sqrt{t}\bpm 1\\ 0\epm & \text{ if } x=\infty,
%\end{dcases}
%\]
%where $B_x$ is the Busemann cocycle centered at $x\in\R\cup\{\infty\}$.

%\subsection{Cofinite Fuchsian groups}
Let $\Gamma<G$ be a discrete subgroup, with respect to the topology induced from $G\subset\R^4$. We will assume that $\Gamma<G$ is a lattice, that is, the quotient $G/\Gamma$ carries a finite $G$-invariant Borel measure, which when normalized to a probability measure we denote by $\mu$. The image of $\G$ in $\PSL_2(\R)$ via the canonical projection is a Fuchsian group, which we also denote by $\G$.

The action of $\Gamma$ by fractional linear transformation on the hyperbolic plane $\h^2$, i.e., $\bsm a&b\\ c&d\esm.z=\tfrac{az+b}{cz+d}$ is properly discontinuous and admits a finite-area fundamental domain in $\h^2$ whose (hyperbolic) area we denote by $V={\rm area}(\G\bk\h^2)$. The area of a geometrically finite Fuchsian group is a numerical invariant that does not depend on the particular choice of fundamental domain; we refer the reader to \cite{Katok} for more properties of Fuchsian groups.

%We will also assume that $-I\in\G$. Note that Fuchsian groups are in one-to-one correspondence with discrete subgroups $\G<G$ that contain $-I$. 

\subsection{Scaling transformations}\label{sec:scaling}
A lattice $\Gamma<G$ is nonuniform if and only if the action of $\Gamma$ on $\h^2$ has parabolic fixed points, called cusps. We say that two cusps $\fa$ and $\fb$ are equivalent if their $\G$-orbits coincide, i.e., if $\G.\fa=\G.\fb$. Since $\G$ is a lattice, there are at most finitely many inequivalent cusps. For each cusp $\fa$, we set $\G_\fa=\{\g\in \G: \g.\fa=\fa\}$ and note that if $\fa$ and $\fb$ are equivalent cusps with $\fb =\g.\fa$, then $\G_\fb=\g\G_\fa \g^{-1}$. 

Computations are most convenient at the cusp at $\infty$; up to conjugation by some $\gs_\fa\in G$ such that $\gs_\fa.\infty=\fa$, we may assume that $\G$ has a cusp at $\infty$ (up to replacing $\G$ by $\gs_\fa^{-1}\G\gs_\fa$). If $-I\not\in\G$, then $\G_\infty$ is isomorphic to $\Z$ and generated by a matrix of the form $\pm \bsm 1&\omega\\0&1\esm$, where we call $\omega>0$ the {\bf cusp width}, while if $-I\in\G$, $\G_\infty$ is isomorphic to $\Z/2\Z\times\Z$. The image of $\G_\infty$ in $\PSL_2(\R)$ is always isomorphic to $\Z$.  The following criterion is well known; see \cite[Lemma 1.7.3.]{Miyake}.

\begin{lm}\label{Miyake}
Let $\g=\bsm a&b\\ c&d\esm\in \G$. If $|c|\omega<1$, then $\g\in\G_\infty$.
\end{lm}

Let $\fa$ be a cusp for $\G$ and choose $\gs_\fa\in  G$ such that 
\[
\gs_\fa.\infty =\fa\quad \text{ and }\quad \sigma_\fa^{-1}\G_\fa\sigma_\fa =  \bbm 1 & \Z\\ 0 &1\ebm,
\]
that is, we fix a cusp at $\infty$ whose width $\omega$ is scaled to 1. We call $\gs_\fa$ a {\bf scaling transformation}. Notice that the above conditions only determine $\gs_\fa$ up to right translation by elements $n\in N$. One has the following convenient double coset decomposition, which can be seen as a local expression of the Bruhat decomposition; see \cite[Theorem 2.7]{Iwa}.

\begin{thm}\label{DCD}
Let $\fa$, $\fb$ be cusps for $\G<\PSL_2(\R)$. We have a disjoint union
\begin{align}\label{dcd}
\sigma_\fa^{-1}\Gamma\sigma_\fb = \delta_\fab \bbm 1 & \Z\\ 0 &1\ebm \cup \bigcup_{\substack{0< a\leq c}} \bbm 1 & \Z\\ 0&1\ebm \bbm a &*\\ c&*\ebm \bbm 1 & \Z\\ 0 & 1\ebm,
\end{align}
where $\delta_\fab=1$ if $\fa$, $\fb$ are $\G$-equivalent and $\delta_\fab=0$ otherwise, and $a$, $c$ run over real numbers such that $\gs_\fa^{-1}\G\gs_\fb$ contains $\bsm a&*\\ c&*\esm$.
\end{thm}

\section{A dichotomy for lattice orbits in the plane} \label{sec:discrete}

Given a nonzero vector $\bx=\bsm x\\ y\esm$ in the Euclidean plane, we say that {\bf the direction of $\bx$ is fixed by a parabolic motion in $G=\SL_2(\R)$} if there is a parabolic element $g\in G$ such that $g.\tfrac{x}{y}=\tfrac{x}{y}$. We have the following (now classical) dichotomy; see 
\cite[Theorem 3.2]{D}.

\begin{thm}
Let $\bx\in\R^2\setminus\{\bo\}$ and let $\G<G$ be a lattice acting linearly on the plane and containing  $-I$. The (linear) orbit $\G\bx$ is either discrete or dense in $\R^2$. More precisely, $\G\bx$ is discrete  if and only if the direction of $\bx$ is fixed by a parabolic motion in $\G$.
\end{thm}
%\begin{proof}
%See \cite[Theorem 3.2]{D} for a proof when $-I\in\G$. If $-I\not\in\G$, we may replace $\G$ by the degree 2 central extension $\G^{(\pm)}$ generated by $-I$ and $\G$. Then $\G^{(\pm)}\bx = \G \bx \cup -\G\bx$. Since $\G\bx\cap -\G\bx$ is either empty or $\G\bx$, we have either $\G^{(\pm)}\bx = \G\bx$ or that $\G^{(\pm)}\bx$ is the disjoint union $\G \bx \cup -\G\bx$. The statement now follows from the case in which $-I\in\G$.
%\end{proof}

In particular, it follows that if the underlying lattice is uniform, each lattice orbit is dense in the plane. The distribution of dense lattice orbits was studied by many authors; see \cite{Greenberg,Led, Gor04, GW, MW12, Kelmer17}. We henceforth assume that $\G$ is nonuniform. Note that the existence of discrete orbits for $\G$ still holds if $-I\not\in\G$; if $\G^{(\pm)}$ denotes the degree 2 central extension generated by $-I$ and $\G$ then $\G^{(\pm)}\bx= \G \bx \cup -\G\bx$ is either a disjoint union or collapses to $\G\bx$ and the previous result applies. The following proposition characterizes further the discrete orbits that arise.

\begin{prop}\label{prop:homothety}
Let $\G<G$ be a nonuniform lattice acting linearly on the plane. Up to homothety, there are only finitely many discrete $\G$-orbits, which are pairwise disjoint and in one-to-one correspondence with the finitely many inequivalent cusps of $\G$. 
\end{prop}

%Remark: $\Lambda_1$, $\Lambda_2$ homothetic if there is $\alpha\in\R^*$ such that $\Lambda_1=\alpha\Lambda_2$. If the lattice contains $-I$ then we may further restrict to $\alpha>0$ and $\Lambda_1$, $\Lambda_2$ are `dilates'.

To see this, we record the following preliminary lemma.

\begin{lm}\label{lm:cusp gps} 
For each $t\in\R\cup\{\infty\}$, set
\[
\bx_t = \begin{dcases} \be_1 & \text{ if } t=\infty,\\ \bsm t\\1\esm & \text{ if } t\in\R. \end{dcases}
\]
Let $\bx=\bsm x\\ y\esm\in\R^2\setminus\{\bo\}$ and let $\G<\SL_2(\R)$ be a lattice acting linearly on $\R^2$. For $\gamma \in \Gamma$ we set $\gamma. t$ to be the image of $\gamma$ under the extension of M\"obius tranformations to the extended real line. Then:
\begin{enumerate}
\item There is a unique $t\in\R\cup\{\infty\}$ such that $\bx$ and $\bx_t$ are collinear;
\item For every $\g\in\G$ and $ t\in\R\cup\{\infty\}$, the vectors $\g \bx_t$ and $\bx_{\g.t}$ are collinear ;
\item For every $s, t\in\R\cup\{\infty\}$ and $\g\in \G$, $\g.s=t$ if and only if $\bx_{\g.s}$ and $\bx_t$ are collinear.
\end{enumerate}
\end{lm}
\begin{proof}[Proof of \cref{lm:cusp gps}]
(1) Clearly, each vector $\bx$ is collinear to $\bx_t$ for $t=\tfrac{x}{y}$ if $y\neq0$ and to $\be_1$ otherwise. Moreover, if $\bx_s$ and $\bx_t$ are collinear, then $s=t$. 

(2) Let $\g=\bsm a&b\\c&d\esm$ and $t\neq\infty$. If $ct+d\neq0$, then $\g \bx_t$ and $\bx_{\g.t}$ are collinear. If $ct+d=0$, then $\g.t=\infty$, and  $\g \bx_t$ is collinear to $\be_1=\bx_{\g.t}$. Suppose now that $t=\infty$. Then $\g\be_1=\bsm a\\ c\esm$ while $\bx_{\g.\infty}=\bx_{a/c}$. 

The ``only if'' part of (3) is immediate. For the converse, we know by (2) that $\bx_t$ is collinear to $\g \bx_s$. Suppose first that $s,t\neq\infty$; we have a vector identity of the form $\bsm t\\ 1\esm = \lambda \bsm as+b\\ cs+d\esm$ for some $\lambda\neq0$. We deduce that $\g.s=t$. If instead, $s=t=\infty$, we want to show that $\g\in\G_\infty$. This then follows from the vector identity $\be_1=\lambda \bsm a\\ c\esm$; since $c=0$, we have $\g\in\G_\infty$. The remaining cases follow along the same lines.
\end{proof}

\begin{proof}[Proof of \cref{prop:homothety}]
If the direction of $\bx=\bsm x\\y\esm$ is fixed by a parabolic element $\g\in\G$, we have that $\bx$ is collinear to $\bx_\fa$, whereby $\fa=\tfrac{x}{y}$ is a cusp for $\G$ and thus the lattice orbits $\G\bx$ and $\G\bx_\fa$ are homothetic. Moreover, if $\fa$ and $\fb$ are equivalent cusps, then by \cref{lm:cusp gps}, the orbits $\G\bx_\fa$ and $\G\bx_\fb$ are homothetic. Hence up to homothety, there are only finitely many discrete $\G$-orbits, in one-to-one correspondence with the finitely many inequivalent cusps of $\G$. 

Suppose that $\bx\in \G\bx_\fa\cap \G\bx_\fb$; say $\bx=\g_1 \bx_\fa = \g_2 \bx_\fb$. Then $\bx_\fa = \g_1^{-1}\g_2 \bx_\fb$ is collinear to $\bx_{\g^{-1}_1 \g_2 \fb}$ and this implies that $\fa$ and $\fb$ are equivalent.
\end{proof}

%\begin{prop}
%Let $M$ be a Veech surface with Veech group $\G_M<G$. Then $S_M$, the set of holonomy vectors of saddle connections, is a finite disjoint union of discrete $\G_M$-orbits.
%\end{prop}
%
%\begin{proof}
%The fact that $S_M$ decomposes into a finite union of discrete $\G_M$-orbits is shown in \cite{Veech89}. We prove that the orbits in this union are pairwise disjoint. We first note that $-I\in\G_M$; both flowing forwards and backwards will yield saddle connections.
%
%The last line of the previous proof extends immediately to show that two discrete orbits associated to inequivalent cusps are necessarily disjoint. Suppose now $\by\in \lambda_1 \G_M \bx_\fa \cap \lambda_2 \G_M\bx_\fa$ with $\lambda_1,\lambda_2>0$ and $\lambda_1\neq\lambda_2$. Thus there is some $\g\in\G_M$ such that $\bx_\fa = \tfrac{\lambda_2}{\lambda_1}\g\bx_\fa$. In fact, up to changing accordingly the scaling factors $\lambda_i$, we may replace $\bx_\fa$ by $\gs_\fa \be_1$, where $\gs_\fa$ is a choice of scaling transformation for the cusp $\fa$. Then the equation above becomes $\be_1 =\tfrac{\lambda_2}{\lambda_1} \g' \be_1$ for some $\g'\in \gs_\fa^{-1}\G_M\gs_\fa$ of the form
%\[
%\bpm \lambda_1/\lambda_2 &* \\ 0&* \epm.
%\]
%Since $\gs_\fa$ is a scaling transformation, the cusp width at infinity for $\gs_\fa^{-1}\G_M\gs_\fa$ is $\omega=1$, and by \cref{Miyake} this forces $\lambda_1=\lambda_2$, contradicting our assumption.
%\end{proof} 

\section{Counting results for discrete lattice orbits}\label{sec:prevcount}

Given a discrete orbit $\Lambda=\G\bx$ as above, we wish to understand its distribution in the plane, starting with a precise estimate for
\[
N_R(\Lambda) = |\Lambda\cap B_R|
\] 
as $R\to\infty$, where $B_R=\{\bx\in\R^2 : \|\bx\|<R\}$. 

We will first choose preferred representatives in each homothety class of discrete $\G$-orbits. Then, up to computing a scaling factor, it suffices to restrict the counting problem to what we will call scaled discrete lattice orbits. Let $\fa_1,\dots,\fa_h\in\R\cup\{\infty\}$ be a set of representatives for the inequivalent cusps of $\G$. Then for each cusp representative $\fa_i$, we set
\begin{align}\label{scaled DLO}
\Lambda_{\fa_i} := \G \gs_{\fa_i} \be_1,
\end{align}
where $\gs_{\fa_i}$ is a scaling transformation for the cusp $\fa_i$ as defined in \cref{sec:scaling}. By \cref{prop:homothety}, each discrete $\G$-orbit is homothetic to one of $\Lambda_{\fa_1},\dots,\Lambda_{\fa_h}$. We call $\Lambda_\fa$ as in (\ref{scaled DLO}) the {\bf scaled $\G$-orbit attached to $\fa$}. Observe that $\Lambda_\fa$ does not depend on the particular choice of scaling $\gs_\fa$ or of cusp representative. Now if $\Lambda=\lambda\Lambda_\fa$, the counting function scales by $N_R(\Lambda) = N_{R/\lambda}(\Lambda_\fa)$. The next proposition provides an algebraic formula to compute the scaling $\lambda$. (For a geometric interpretation of $\lambda$, see \cite[Theorem 6.1]{GJ}.)

\begin{prop}\label{GJ}
Given a discrete orbit $\G\bx$ associated to cusp $\fa$ of $\G$, we have $\G\bx=\lambda \Lambda_\fa$ with
\[
\lambda = \frac{\|\bx\|}{\sqrt{\mathrm{tr}(S\g_\fa)}},
\]
where $S=\bsm 0&-1\\1&0\esm$ and $\g_\fa$ denotes the cyclic generator of $\G_\fa$. 
\end{prop}

\begin{proof}
There exists some $g\in G$ such that $\bx=g\be_1$. Then $g^{-1}\g_\fa g=\bsm 1&\omega\\ 0&1\esm$ for some $\omega>0$. We may choose $\gs_\fa=  g a_{\sqrt\omega}$ as scaling transformation; it then follows that $\G\bx = \tfrac{1}{\sqrt\omega}\Lambda_\fa$. It remains to compute the cusp width $\omega$. Write $\bx=\bsm x\\ y\esm$. Then $g = \bsm x&*\\ y&*\esm$ and a direct computation shows that $\g_\fa= g\bsm 1&\omega\\0&1\esm g^{-1} = \bsm * & x^2 \omega\\ -y^2 \omega &*\esm$,  $S\g_\fa = \bsm y^2\omega &*\\ *& x^2 \omega\esm$ and thus $\omega = \tfrac{\text{tr}(S\g_\fa)}{\|\bx\|^2}$.
\end{proof}

The current best counting result for $N_R(\Lambda)$ is given by the following theorem.

\begin{thm}{\cite[Theorem 4.1]{BNRW}} \label{BNRW}
Let $\Lambda$ be a scaled discrete $\G$-orbit, and let $g\in G$. There are numbers $1=\rho_0> \rho_1>\dots>\rho_k>\tfrac12$ and constants $c_1,\ldots, c_k$ such that we have the asymptotic expansion
\begin{align}\label{full asympto}
N_R(g\Lambda)= |g\Lambda\cap B_R| =   c_\Gamma |B_R| + c_1 |B_R|^{\rho_1} + c_2 |B_R|^{\rho_2} +\dots + c_k |B_R|^{\rho_k} + O\left(|B_R|^{2/3}\right),
\end{align}
where the leading constant is given by 
\begin{equation} \label{c_Gamma}
c_\Gamma = \begin{dcases} \frac{2}{\pi V} & \text{ if } -I\in\G,\\ \frac{1}{\pi V} & \text{ if } -I\not\in\G. \end{dcases}
\end{equation}
\end{thm}

The constants $\rho_i\in(\tfrac12,1)$ are eigenparameters of the finitely many possible small residual eigenvalues of the Laplacian, with $\lambda_i=\rho_i(1-\rho_i)\in(0,\tfrac14)$, where $\lambda_i$ belongs to the residual spectrum. Set 
\begin{align}\label{def:delta}
\delta = \min\{ 2(1-\rho_1), 2/3\}.
\end{align}
Then $N_R(g\Lambda) = c_\G \pi R^2 + O(R^{2-\delta})$. There are known constructions of  nonuniform lattices for which $\rho_1$ is arbitrarily close to 1 \cite{Selberg65}.

The small eigenvalues $\lambda_i$, when they arise, are not explicit. It is known that $\SL_2(\Z)$ and its congruence subgroups have no small residual eigenvalues \cite{Iwa} and this is also true for triangle groups \cite[Page 583]{Hejhal06}. In such cases (\ref{full asympto}) reduces to $N_R(g\Lambda)=c_\G |B_R| +O(|B_R|^{2/3})$, where the exponent $2/3$ is an artefact of the method of proof. %; it is believed that $2/3$ could be replaced by $1/2+\eps$ and that this is optimal if one wants a statement for all nonuniform lattices in $G$. 
For specific lattices, a better result is nonetheless possible. This is easy to see for congruence subgroups of $\SL_2(\Z)$. For the convenience of the reader, we include a full proof here. 

\begin{thm}\label{thm:cong}
Let $\G$ be a congruence subgroup of $\SL_2(\Z)$, and let $\Lambda$ be a scaled discrete $\G$-orbit. Then
\[
N_R(\Lambda) = c_\G |B_R| + o(|B_R|^{1/2}).
\]
Assuming the Riemann hypothesis (RH), the error term $o(|B_R|^{1/2})$ can be replaced by $O(|B_R|^{5/12+\eps})$.
\end{thm}

\begin{proof}
Recall that $\G$ is a congruence subgroup of $\SL_2(\Z)$ if it contains a principal congruence subgroup $\G(N)=\ker\left(\SL_2(\Z)\to\SL_2(\Z/N\Z)\right)$ for some $N\geq1$. We have a finite disjoint decomposition $\G=\cup\tau\G(N)$ for a choice of coset representatives $\tau\in \G/\G(N)$. This implies 
\[
N_R(\G\bx) = \sum_\tau N_R(\tau\G(N)\bx).
\]
We will consider each summand separately.

Since $\Lambda$ is assumed to be scaled, we have $\bx=\gs_\fa \be_1$ for some choice of scaling transformation $\gs_\fa$. Each cusp of $\G$ is a cusp of $\G(1)=\SL_2(\Z)$ and hence we may choose $\gs\in\G(1)$ such that $\gs(\infty)=\fa$ and set $\gs_\fa=\gs a_{\sqrt\omega}$, where $\omega$ is the cusp width of $\gs^{-1}\G\gs$. Then 
\[
N_R(\tau \G(N)\bx) = N_R (\tau \G(N) \gs a_{\sqrt\omega} \be_1) = N_R( \tau \gs \G(N)a_{\sqrt\omega} \be_1),
\]
where we used that $\G(N)$ is a normal subgroup of $\G(1)$. For readability we write $A= \tau\gs$. Then
%
%The set of cusps of $\G$ coincides with the set of cusps of $\G(N)$ (see, e.g., \cite[Corollary 1.5.5]{Miyake}) so that if $\G\bx$ is a discrete orbit corresponding to a cusp $\fa$ of $\G$, then $\G(N)\bx$ is again a discrete orbit corresponding to $\fa$. Since $\G(1)=\SL_2(\Z)$ has only one class of inequivalent cusps, all cusps of $\G(N)$ are $\G(1)$-equivalent. Choose $\gs\in\G(1)$ such that $\gs.\infty=\fa$. Since $\G(N)$ is a normal subgroup of $\G(1)$, we have $\gs^{-1}\G(N)\gs=\G(N)$ and $\gs^{-1}\G(N)_\fa \gs =  \bsm 1 &N\Z \\0&1\esm$. Then by \cref{lm:cusp gps}, the discrete lattice orbits $\G(N)\bx$ and $\G(N)\gs\be_1=\gs\G(N)\be_1$ are homothetic. Set $\lambda\in\R^*$ such that $\G(N)\bx =\lambda \gs\G(N)\be_1$. Then 
\begin{align}\label{a}
N_R(A\G(N)a_{\sqrt\omega}\be_1)  &= |\{\g\in\G(N)/a_{\sqrt\omega}\G(N)_\infty a^{-1}_{\sqrt\omega} : \g\be_1\in \omega^{-1/2}A^{-1}B_{R}\}|\nonumber\\
&= \frac{\omega}{N} |\{\xi\in\Z^2_{\rm prim} :\xi \equiv \be_1\, (\text{mod }N),\, \xi\in \omega^{-1/2}A^{-1}B_{R}\}|.
\end{align}
Here $\Z^2_{\rm prim}$ denotes the set of all vectors $\xi=\bsm m\\ n\esm$ with coprime integer coordinates $(m,n)=1$, and $\omega^{-1/2}A^{-1}B_R$ is an ellipsoid centered at the origin with area $\tfrac{\pi}{\omega}R^2$.

We have
\begin{align*}
\frac{N}{\omega}\cdot (\ref{a}) &=
\sum_{\substack{(m,n)\in\Z^2_{\rm prim}\\ m\equiv 1\, (N)\\ N\mid n}} \one_{\omega^{-1/2}A^{-1}B_R}((m,n)) =
\sum_{\substack{d\geq1\\ (d,N)=1}} \mu(d) \sum_{\substack{(m,n)\in\Z^2\\ m\equiv 1\, (N)\\ d\mid m,\, dN\mid n}} \one_{\omega^{-1/2}A^{-1}B_R}((m,n)).
\end{align*}
where $\mu(n)$ is the M\"obius function and we used that $\sum_{d\mid (m,n)} \mu(d)=1$ iff $(m,n)=1$ (and is 0 otherwise). By the Chinese remainder theorem, the system of congruence equations for $m$ has a unique solution $m_0$ in $[0,dN)$ and thus
\begin{align*}
\frac{N}{\omega}\cdot(\ref{a}) &=
\sum_{\substack{d\geq1\\ (d,N)=1}} \mu(d) \sum_{\substack{\xi\in\Z^2\\ \xi\equiv \bsm m_0\\ 0\esm\, (dN)}} \one_{\omega^{-1/2}A^{-1}B_R}(\xi)\\
& = \sum_{\substack{d\geq1\\ (d,N)=1}} \mu(d) \sum_{\xi\in\Z^2} \one_{\tfrac{1}{dN}(\omega^{-1/2}A^{-1}B_R-m_0\be_1)}(\xi).
\end{align*}

For $R$ sufficiently large, $\Omega=\tfrac{1}{dN}(\omega^{-1/2}A^{-1}B_R- m_0\be_1)$ is a compact convex planar set of area 
\[
|\Omega| = \frac{|B_R|}{\omega(dN)^2}
\]
 that contains the origin as an inner point. A simple geometric argument shows that
\begin{align}\label{Gauss}
\sum_{\xi\in\Z^2} \one_{\Omega}(\xi) = \frac{|B_R|}{\omega(dN)^2} + O_{\omega,N}\op{\frac{R}{d}}.
\end{align}
Using available estimates towards the Gauss circle problem in convex planar domains, the error term can further be replaced by $O(R^{46/73}\log^{315/146}(R))$ \cite[Theorem 5]{Huxley}. Let $\rho>1/2$. To summarize we have, formally,
\begin{align*}
(\ref{a}) &= \frac{1}{N^3} \sum_{(d,N)=1} \frac{\mu(d)}{d^2} |B_R| + O\op{\sum_{d\geq1}\frac{\mu(d)}{d^\rho}|B_R|^{\rho/2}}
\end{align*}
with convergence guaranteed when $\rho>1$. Via the prime number theorem, the error term can be replaced with $o(|B_R|^{1/2})$ and under RH by $O(|B_R|^{5/12+\eps})$; see \cite{HuxleyNowak1996}.  We conclude that
\[
N_R(\Lambda) = \frac{[\G:\G(N)]}{N^3} \sum_{(d,N)=1}\frac{\mu(d)}{d^2} |B_R| + o(|B_R|^{1/2})
\]
(unconditionally on RH or with $O(|B_R|^{5/12+\eps})$ under RH). The leading constant can be shown to equal $c_\G$; see 
\cite[Theorem 4.2.5]{Miyake}. 
\end{proof}

\section{Theta transforms}\label{sec:theta}

Let $k>0$ and let $\Lambda_1,\dots,\Lambda_k$ be $k$ discrete $\G$-orbits. Let $f$ be a Borel measurable function on $(\R^2)^k=\R^2\times\R^2\times\cdots\times\R^2$. We define the corresponding theta/Siegel--Veech transform to be
\[
\Theta=\Theta_{\Lambda_1,\dots,\Lambda_k;f}:G/\G\to\R,\quad \Theta(g\G) = \sum_{(\bx_1,\dots, \bx_k)\in\Lambda_1\times\cdots\times\Lambda_k} f(g\bx_1,\dots,g\bx_k).
\]

\begin{lm}\label{highbound}
If $f$ is bounded with compact support, then $\Theta\in L^\infty(G/\G)$.
\end{lm}

\begin{proof}
Since $f$ is bounded and has compact support, for all $(\bx_1,\ldots,\bx_k) \in (\R^2)^k$, and for some constant $c_f, R >0$, $f(\bx_1,\ldots, \bx_k)\leq c_f \one_{B_{R}}(\bx_1) \cdots\one_{B_{R}}(\bx_k)$.  By \cite[Lemma 16.10]{Veech98}, since $\one_{B_R}$ is bounded with compact support, $\Theta_{\Lambda_i;\one_{B_R}}$ is uniformly bounded over $G/\Gamma$. Thus since $\mu$ is a probability measure, we conclude that $\Theta_{\Lambda_i; \one_{B_R}} \in L^\infty(G/\Gamma)$. We then notice that $\Theta_{\Lambda_1,\dots,\Lambda_k; f} \leq c_f \prod_{i=1}^k \Theta_{\Lambda_i;\one_{B_R}}\in L^\infty(G/\Gamma)$ since $G/\G$ is a probability space.
\end{proof}

Recall that a function $f:\R^n\to\R\cup\{\infty\}$ is {\bf lower (resp.~upper) semi-continuous at $\bx_0$} if $\liminf_{\bx\to \bx_0} f(\bx) \geq f(\bx_0)$ respectively $\limsup_{\bx\to\bx_0} f(\bx)\leq f(\bx_0)$. Characteristic functions of open (resp.~closed) sets are lower (resp.~upper) semi-continuous.  Semi-continuous behave nicely with respect to monotone approximation; see Baire's theorem or \cref{rmk-SC} below.

 \begin{defn}
We denote $SC((\R^2)^k)$ to be the space of all lower semi-continuous functions bounded below, and all upper semi-continuous functions bounded above. 
\end{defn}

\begin{lm}\label{repthm}
There exists a unit regular $G$-invariant Borel measure $\nu$ on $(\R^2)^k$ such that for any $f\in SC((\R^2)^k)$, we have
\begin{equation}\label{rmkrt}
\int_{G/\G} \Theta_{\Lambda_1,\dots,\Lambda_k;f}(g)d\mu(g) = \int_{(\R^2)^k} f(\bx) \,d\nu(\bx).
\end{equation}
\end{lm}

\begin{proof}
By \cref{highbound}, for $f\in C_c((\R^2)^k)$, we have $\Theta_{\Lambda_1,\dots,\Lambda_k;f}  \in L^\infty(G/\G)$. Thus the assignment $f \mapsto \int_{G/\G} \Theta_{\Lambda_1,\dots,\Lambda_k;f}(g)d \mu(g)$ defines a $G$-invariant positive linear functional on $C_c((\R^2)^k)$ where $G$ acts diagonally on $(\R^2)^k$. Hence by the Riesz--Markov--Kakutani representation theorem, there is a unit regular $G$-invariant Borel measure $\nu$ on $\R^2$ such that \cref{rmkrt} holds for $f\in C_c((\R^2)^k)$. 

Every lower semi-continuous function $f$ bounded below can be approximated by a non-decreasing sequence $(f_j)\subset C_c((\R^2)^k)$ that converges pointwise to $f$. Since $f$ is bounded from below, it suffices to consider non-negative semi-continuous functions. By Baire's thereom, $\tilde{f}$ is the pointwise limit of continuous functions $h_k$ which are monotonely non-increasing. Choosing a continuous bump function $\tilde{h}_k$ to be $1$ on $B(0,k)$ and $0$ outside $B(0,k+1)$, the sequence $h_k \tilde{h}_k$ is the desired sequence of continuous functions of compact support. If $f$ is upper-semicontinuous and bounded above, then $-f$ is lower semi-continuous and bounded below, reducing to case of lower semi-continuous functions.

Moreover pointwise monotone convergence implies $\Theta_{\Lambda_1,\dots,\Lambda_k;f_j}$ also monotonically increases to $\Theta_{\Lambda_1,\dots,\Lambda_k;f}$ pointwise for each $g \in G/\G$.  We can then apply the monotone convergence theorem. We emphasize here that we use pointwise convergence (in $(\R^2)^k$) in order to automatically guarantee pointwise convergence (in $G/\G$) under the transformation $f \mapsto \Theta_{\Lambda_1,\dots,\Lambda_k;f}$.
 \end{proof}

\begin{rmk}\label{rmk-SC}
The proof makes use of the fact that the dual of functions $f \in C_c((\R^2)^k)$ which are continuous with compact support  are exactly the Radon measures on $(\R^2)^k$. It is not straightforward to extend this representation from $f\in C_c((\R^2)^k)$ to all integrable Borel functions. For example, if $A=\{x:1\leq |x|<2\}$, the characteristic function $\one_A$ cannot be pointwise approximated from above or below by sequences of continuous functions. But lower (or upper) semicontinuous functions precisely have this property. %Thus we will state the representation theorem for $SC((\R^2)^k))$, which we define to be the space of all lower and upper semicontinuous functions. In our applications we will use the fact that $SC((\R^2)^k)$ contains all characteristic functions of closed or open sets.
\end{rmk}

\section{Veech's mean value theorem}\label{sec:Veech}

For simplicity of exposition, we will from here on state our results for scaled discrete lattice orbits. Recall that each discrete $\G$-orbit $\Lambda_i$ is homothetic to some scaled orbit $\Lambda_{\fa_i}$; explicitly, we write $\Lambda_i =\lambda_i \Lambda_{\fa_i}$ and set ${\bf \lambda} = (\lambda_1,\dots,\lambda_k)\in \R^k$. Then
\[
\Theta_{\Lambda_1,\dots,\Lambda_k;f} = \Theta_{\Lambda_{\fa_1},\dots,\Lambda_{\fa_k};f\circ{\bf \lambda}},
\]
where $f\circ{\bf \lambda}(\bx_1,\dots,\bx_k)=f(\lambda_1\bx_1,\dots,\lambda_k \bx_k)$.

Let $\Lambda$ be a scaled discrete $\G$-orbit. The following mean value theorem, extending the classical theorem of Siegel \cite{Siegel} in the geometry of numbers, is due to Veech \cite{Veech98}.

\begin{thm}\cite[Theorem 6.5]{Veech98}\label{Veech}
Let $\Lambda$ be a scaled discrete $\G$-orbit. For each integrable Borel function $f:\R^2 \to \R$, the Siegel--Veech transform $\Theta_{\Lambda;f}$ satisfies
\begin{equation} \label{eq:SV}
\int_{G/\G} \Theta_{\Lambda;f}(g)d\mu(g) = c_\Gamma \int_{\R^2} f(\bx) d\bx,
\end{equation}
with $c_\G$ as defined in \cref{c_Gamma}.
\end{thm}

\begin{rmk}
For a general (i.e., not scaled) discrete $\G$-orbit $\Lambda$, the constant $c_\G$ is simply replaced by $c_\G/\lambda^2$ with $\lambda$ as defined in \cref{GJ}.
\end{rmk} 

For completeness, we include a proof of Veech's theorem. The statement is fairly immediate for semicontinuous functions, but the extension to integrable Borel functions stated by Veech actually requires the following technical lemma. Here $\lambda_n$ denotes the Lebesgue measure on $\R^n$.

\begin{lm} \label{lem:null}
If $\{f_{k}\}_{k=1}^{\infty}, f$ are Borel functions in $L^{1}(\R^{2},\lambda_2)$ with $f_{k}(\by) \to f(\by)$ for $\lambda$-a.e.~$\by$, then $\Theta_{\Lambda; f_k}(g) \to \Theta_{\Lambda;f}(g)$ for $\mu$-a.e.~$g \in G/\G$. 
\end{lm}

\begin{proof}
Recall the definition (\ref{mathscr-S}) of the set $\mathscr{S}$ and let $\mathscr{Y}\subset \mathscr{S}$ correspond to $\mathscr{G}= \{ g\in G: \Theta_{\Lambda;f_k}(g)\not\to\Theta_{\Lambda;f}(g)\}$ up to a set of measure zero. Let $Y=\{\by\in\R^2: f_k(\by)\not\to f(\by)\}$. The set $\mathscr{Y}$ is contained in 
\begin{align*}
\bigcup_{\bx\in\Lambda} \{ &g(a,b,s) \bx\in Y\}\\
& = \bigcup_{\bx\in\Lambda}\bigcup_{k\geq1} \{g(a,b,s)\bx\in Y: a\in[-2^k, -2^{-k}] \cup[2^{-k},2^k],\ (b,s)\in[-2^k,2^k]^2\}.
\end{align*}
We denote the latter sets by $\mathscr{X}_\bx^k \subseteq \R^3$. Fix $\bx\in\R^2\setminus\{\bo\}$, and $k\in\N$. We will show that 
\[
\lambda_3(\mathscr{X}_\bx^k)=0.
\] 
By the coarea formula for Lipschitz functions \cite{Federer} and our hypothesis $\lambda_2(Y)=0$, for $J$ the Jacobian of the map $(a,b,c) \mapsto g(a,b,c)$ and $\mathcal{H}^1$ the Hausdorff 1-measure, we have
\[
\int_{\mathscr{X}_\bx^k} J(g(a,b,s)\bx) d\lambda_3(a,b,s) = \int_{Y} \mathscr{H}^1(\{a,b,s): g(a,b,s)\bx=\by\})d\lambda_2(\by) = 0.
\]
Hence $\lambda_3(\mathscr{X}_\bx^k)=0$ if and only if $\{J(g\bx)=0\}\cap\mathscr{X}_\bx^k$ has measure zero. In fact, we can directly compute that 
\[D(g(\cdot,\cdot,\cdot)\bx)(a,b,s) = \begin{pmatrix} x_1 &x_2 & 0\\ sx_1- a^{-2}x_2 & sx_2 & ax_1 + bx_2\end{pmatrix}.\]
Since $\bx \neq \bo$, and $a\neq 0$, we in fact have $D(g\bx)$ has rank $2$, so $D(g\bx) [D(g\bx)]^T$ has full rank and hence   the Jacobian $(J(g\bx))^2 = \det(D(g\bx) [D(g\bx)]^T) \neq 0$. 

We have now shown $\lambda_3(\mathscr{X}_\bx^k)=0$, which implies that $\eta(\mathscr{G})=0$. To conclude the proof, we note that since $\mu$ is the probability measure when $\eta$ is restricted to a finite volume fundamental domain associated to $G/\G$, we in fact have convergence $\Theta_{\Lambda;f_k}(g) \to\Theta_{\Lambda;f}(g)$ for $\mu$-a.e. $g\in G/\G$.
\end{proof}

\begin{proof}[Proof of \cref{Veech}]
Let $f\in SC(\R^2)$. Then \cref{repthm} applies, and we seek to determine $\nu$. The plane $\R^2$ decomposes into the disjoint $G$-orbits $\{\bo\}\cup(\R^2\setminus\{\bo\})$ and hence the only $G$-invariant measures on $\R^2$ (up to scaling) are $\delta_0$, the point mass at the origin, and the Lebesgue measure $ \lambda$. Since the Haar measure on $G$ is unique up to scaling, we conclude that $\nu= \alpha \delta_0 + \beta \lambda$ for some real constants $\alpha, \beta\geq0$. 

We first show that $\alpha =0$. Indeed since $f = \one_{\bo}$ is a characteristic function of a closed set, then $\Lambda\subset\R^2\bk\{\bo\}$ implies that $\alpha = 0$. 

Let $f$ be the characteristic function of the open disk of radius $R$ centered at the origin, and recall that characteristic functions of open sets are lower semi-continuous. Hence for $R>0$, we are left with
\[
\beta = \int_{G/\G} \frac{|g\Lambda\cap B_R|}{|B_R|}d\mu(g).
\]
Letting $R\to\infty$, we conclude with \cref{BNRW} that $\beta=c_\Gamma$. 

To extend the statement from semicontinuous functions to integrable Borel functions, we first consider the class of bounded Borel functions with compact support. With \cref{lem:null} in hand, we can use the dominated convergence theorem (where pointwise almost everywhere convergence is needed on both the left and right hand side of (\ref{eq:SV}))  to bootstrap our way to any integrable Borel function.
%{With \cref{lem:null} in hand, using a standard approximation argument of $f\in L^1$ by some $(f_k)\subset C_c^\infty$ and dominated convergence, we can extend (\ref{eq:SV}) from $f\in C_c(\R^2)$ to $f\in L^1(\R^2)$. This argument does not yet extend to the family of nonnegative Borel functions, for which we stated \cref{Veech}. Before describing the proof, we note that if $f$ is nonnegative Borel but not integrable then $\Theta_{\Lambda;f}\not\in L^1(G/\G,\mu)$.}
By \cref{highbound} and \cref{lem:null}, we apply dominated convergence on both sides of \cref{eq:SV} to a sequence of semi-continuous functions converging a.e. to any bounded Borel function with compact support. 

We now extend to all nonnegative Borel functions. Consider $A_k=\{|\bx|\leq 2^k\}\cap\{f(\bx)\leq 2^k\}$ and define $
f_k(\bx)\defeq f(\bx) \one_{A_k}(\bx).$ Monotone convergence implies that 
\[
\lim_{k\to\infty} \int_{\R^2} f_k(\bx)d\bx\ = \int_{\R^2} f(\bx)d\bx
\]
(even when both sides are infinite). Each $f_k(\bx)$ is a bounded Borel function with compact support, hence $L^1$ so that (\ref{eq:SV}) holds for each $k$. Clearly $f_k\leq f_{k+1}$ implies that $\Theta_{\Lambda;f_k}\leq \Theta_{\Lambda;f_{k+1}}$ for each $k$. Hence, by \cref{lem:null} and monotone convergence, we have
\[
\int_{G/\G}\Theta_{\Lambda;f}(g)d\mu(g) = \lim_{k\to\infty} \int_{G/\G} \Theta_{\Lambda;f_k}(g)d\mu(g)\ =\ c_\Gamma \lim_{k\to\infty} \int_{\R^2} f_k(\bx)d\bx = c_\G \int_{\R^2} f(\bx)d\bx.
\]
Finally for any integrable Borel function $f$, we decompose $f$ into its positive and negative parts to complete the proof of \cref{Veech}.
\end{proof}

\section{Admissible determinants}\label{sec:Admissible}

Our goal is to extend \cref{Veech} to an integral formula for pairs of discrete lattice orbits. Similarly to the proof of Veech, our formula will build on the decomposition of the space $\R^2\times\R^2$ into the union of disjoint $G$-orbits given by
\begin{align*}
\R^2\times\R^2 &= \{\bo\}\cup(\{\bo\}\times\R^2)\cup (\R^2\times\{\bo\})\cup\{(\bx,\by):\det(\bx\mid\by)=0\}\cup \{(\bx,\by): \det(\bx\mid\by)\neq0\}\\
&= G(\bo,\bo) \cup G(\bo,\be_1)\cup G(\be_1,\bo)\cup \{G(\be_1,t\be_1) : t\in\R\} \cup \{G(\be_1, c\be_2) : c\in\R^*\}.
\end{align*}
where we write $\det(\bx\mid\by)$ to be the determinant of the matrix with ordered columns $\bx$, $\by$. In particular, we will need to understand the set of determinants $\det(\bx\mid\by)$ arising from ordered pairs $(\bx,\by)\in\Lambda_1\times\Lambda_2$. We record necessary preliminary results in this section. The mean value formula for pairs of discrete lattice orbits will be derived in the next section.

In this section, we fix two (not necessarily distinct) scaled discrete lattice orbits $\Lambda_\fa$, $\Lambda_\fb$. Let 
\[
\cN_\fab =\{c\in\R: \text{there are } \bx\in\Lambda_\fa, \by\in\Lambda_\fb \text{ such that } \det(\bx\mid\by)=c\}
\]
be the set of {\bf admissible determinants} for vector pairs in $\Lambda_\fa\times\Lambda_\fb$. Under the assumption that $-I\in\G$, we have $c\in\cN_\fab$ if and only if $-c\in\cN_\fab$.

\begin{lm}\label{determinants and cosets}
The following statements are equivalent:
\begin{enumerate}
\item
$0\in\cN_\fab$;
\item
$\fa$ and $\fb$ are equivalent cusps;
\item
$\Lambda_\fa=\Lambda_\fb$. %and not just homothetic
\end{enumerate}
\end{lm}
\begin{proof}
If  $0\in\cN_\fab$, there is $\g\in\G$ such that $\g\bx_\fa$ and $\bx_\fb$ are collinear, and hence by \cref{lm:cusp gps}, the cusps $\fa$ and $\fb$ are equivalent. Then $\Lambda_\fa=\Lambda_\fb$ (we may take $\gs_\fb=\g\gs_\fa$) and this in turns implies immediately that $0\in \cN_\fab$.
\end{proof}

We introduce an equivalence relation on the set of ordered pairs $(\bx,\by)\in\Lambda_\fa\times\Lambda_\fb$ by declaring $(\bx_1,\by_1)\sim(\bx_2,\by_2)$ iff $(\bx_1,\by_1)=(\g\bx_1,\g\bx_2)$ for some $\g\in\G$. The determinant $\det(\bx\mid \by)$ is preserved by the equivalence relation. We set
\begin{align*}
\varphi_\fab(c) &= |\{[(\bx,\by)]: (\bx,\by)\in\Lambda_\fa\times\Lambda_\fb,\, \det(\bx\mid\by)=c\}|\end{align*}
to be the counting function of equivalence pairs with determinant $c$.

\begin{lm}\label{lemma2}
Suppose that $-I\in\G$ and let $\cN_\fab^*\coloneqq \cN_\fab -\{0\}.$ Then $c\in\cN_\fab^*$ if and only if there exists a vector $\bsm a\\ c\esm\in(\gs_\fa^{-1}\G\gs_\fb)\be_1$ with $0<a\leq |c|$. In particular
\[
\varphi_\fab(c) = \left|\left\{ 0<a\leq |c|: \bpm a\\ c\epm \in (\gs_\fa^{-1}\G\gs_\fb)\be_1\right\}\right|.
\]
\end{lm}
\begin{proof}
Let $(\bx,\by)$ be an ordered pair in $\Lambda_\fa\times\Lambda_\fb$ such that $\det(\bx\mid\by)=c$. We can choose $\bx'$ such that $(\bx\mid \bx')\in \G \gs_\fa$. Then 
\[
(\bx\mid \bx')^{-1}(\bx\mid\by) = \bpm 1& a\\ 0& c\epm.
\]
This shows that $\bsm a\\ c\esm\in \gs_\fa^{-1}\G\by =(\gs_\fa^{-1}\G\gs_\fb)\be_1$. Under the assumption that $-I\in\G$, we have $\bsm 1&1\\0&1\esm\in\gs_\fa^{-1}\G\gs_\fa$. By matrix multiplication on the left with an appropriate power $\bsm 1&m\\0&1\esm$, we have
\[
\bpm 1 & m\\0&1\epm \bpm a  \\c\epm  = \bpm a+mc \\ c\epm
\]
with $0< a+mc \leq |c|$. Hence we may always choose $a$ with this property. Moreover we recover the same vector for any equivalent ordered pair $(\bx_1,\by_1)$. This assignment is easily seen to be surjective. Indeed, given $\bsm a\\ c\esm\in(\gs_\fa^{-1}\G\gs_\fb)\be_1$, choose $\bx=\gs_\fa\be_1$ and $\by=\gs_\fa\bsm a\\ c\esm$. To prove that it is also injective, suppose that we have $(\bx_1\mid\by_1)=(\bx_1\mid \bx_1')(\bx_2\mid \bx'_2)^{-1}(\bx_2\mid \by_2)$. Then $(\bx_1\mid\by_1)$ and $(\bx_2\mid \by_2)$ are equivalent ordered pairs of vectors.
\end{proof}

\begin{rmk}
Observe that $\varphi_\fab(c)$ counts the number of double coset representatives with lower left entry given by $c$; compare to \cref{DCD}. 
\end{rmk}

When $\G=\SL_2(\Z)$, we may take $\gs_\infty=I$ and $\Lambda_\infty=\SL_2(\Z)\be_1$ for the unique cusp at infinity. Then each integer $n\in\Z$ is realized as an admissible determinant, and $\varphi_\infty(n)=\varphi(|n|)$, Euler's totient function, which is trivially bounded above by $\varphi(n)\leq n-1$ for each $n\in\N$. For more general groups, and in particular groups that are not finite-index subgroups of $\SL_2(\Z)$, one only has the estimate $|\varphi_\fab(c)|\ll |c|^2$ as consequence of the pigeonhole principle; see \cite[Proposition 2.8]{Iwa}. However, we have stronger bounds on average via the scattering matrix for $\G$, which we quickly recall. If $h$ is the number of inequivalent cusps for $\G$, the scattering matrix $(\phi_\fab(s))_{\fa,\fb}$ for $\G$ is the $h\times h$ matrix with entries given by
\[
\phi_\fab(s) = \sqrt\pi \frac{\G(s-1/2)}{\G(s)} \sum_{\substack{c\in\cN_\fab^*\\c>0}} \frac{\varphi_\fab(c)}{c^{2s}},
\]
where $s\in\C$ and $\G(s)$ is Euler's gamma-function. Each entry converges absolutely and uniformly on compact subsets in the half-plane $\re(s)>1$ and admits a meromorphic continuation to the whole complex plane. (See \cite{Iwa} for these facts as well as the ones used in the proof below.)

\begin{thm}\label{Good}
Let $\G<G$ be a nonuniform lattice containing $-I$ with cusps $\fa, \fb$. There are constants $1>\rho_1>\dots>\rho_k>\tfrac12$ and $c_1,\dots,c_k$ such that
\[
\sum_{\substack{c\in\mathcal{N}^*_\fab\\ 0<c< T}} \varphi_\fab(c) = \frac{c_\G}{2}  T^2 + c_1 T^{2\rho_1} + \dots +c_k T^{2\rho_k} + O(T^{4/3})
\]
as $T\to\infty$.%, where $\delta_\fab=1$ if $\fa$ and $\fb$ are equivalent cusps and 0 otherwise.
\end{thm}

\begin{proof}
The asymptotic is a special case of a more general theorem of Good (see \cite[Theorem 4, p.~116]{Good}) on counting functions over double cosets of Fuchsian groups. We include a much simpler and direct proof for the convenience of the reader that recovers the weaker error growth rate $O(T^{4/3+\eps})$. 

We will need the following facts pertaining to the meromorphic continuation of the entries of the scattering matrix, which we also will denote by $\phi_\fab(s)$. We record that $\phi_\fab(s)$ is holomorphic for $\re(s)\geq1/2$ except for at most finitely many poles $\rho_0, \rho_1,\dots,\rho_k$, all simple, and all lying on the real segment $(\tfrac12,1]$ with $\rho_0=1$ whose residue is
\[
\underset{s=1}{{\rm Res}}\, \phi_\fab(s) = \frac{1}{V}.
\]

We will also rely on the following standard complex analysis computation; uniformly for $y ,\gs>0$ and $X\geq1$ we have
\[
\frac{1}{2\pi i} \int_{\gs-iX}^{\gs+iX} y^s \frac{ds}{s} = \begin{dcases} 1+O\op{\tfrac{y^{\gs}}{X|\log y|}} &\text{ if } y>1,\\ O\op{\tfrac{y^{\gs}}{X\log y}} &\text{ if } y<1.
\end{dcases}
\]
We may assume that $T\not\in\cN_\fab$ and let $\gs>1$. Then 
\begin{align*}
\frac{\pi^{-1/2}}{2\pi i}\int_{\gs-iX}^{\gs+iX} \frac{\G(s)}{\G(s-1/2)}\phi_\fab(s)\frac{T^{2s}}{s} ds 
%&= \sum_{\substack{c\in\cN_\fab\\ c>0}} \frac{\varphi_\fab(c)}{2\pi i} \int_{\gs-iX}^{\gs+iX} \op{\frac{T^2}{c^2}}^s \frac{ds}{s} 
= \sum_{\substack{c\in\cN_\fab\\ 0<c<T}} \varphi_\fab(c) + O\op{\frac{T^{2\gs}}{X}}.
\end{align*}
Applying Stirling's estimate for the gamma-function and the Phr\'agmen--Lindel\"of principle we see that the integrand on the LHS is of order $O\op{\tfrac{T^{2\gs}}{|t|^{1/2}}}$ for all $s\in\C$ with $\re(s)\in[1/2,\gs]$ and $|t|\geq1$. To compute the integral on the LHS, we apply Cauchy's residue theorem to the rectangle with sides $[\gs-iX,\gs+iX],\, [\gs+iX,1/2+iX],\, [1/2+iX,1/2-iX],\, [1/2-iX,\gs-iX]$, and find that the LHS is
\[
= \sum_{i=0}^k \frac{\G(\rho_i)\pi^{-1/2}}{\rho_i \G(\rho_i-1/2)}{\rm Res}_{s=\rho_i}\phi_\fab(s) T^{2\rho_i} + O\op{TX^{1/2} + X^{-1/2}\frac{T^{2\gs}}{\log T}}.
\] 
Combining these results, we have 
\[
 \sum_{\substack{c\in\cN_\fab\\ 0<c<T}} \varphi_\fab(c) = \frac{T^2}{\pi V} + c_1 T^{2\rho_1} +\dots + c_k T^{2\rho_k} + O\op{\frac{T^{2\gs}}{X} + TX^{1/2}+X^{-1/2} \frac{T^{2\gs}}{\log T}},
\]
for the appropriate choice of $c_1,\dots,c_k$. The error term is optimized by choosing $X=T^{\gs/3}$ and $\gs$ can be chosen as small as $\gs=1+\eps$ for any $\eps>0$.
\end{proof}

%\begin{rmk}
%The argument given here can be upgraded to remove the extra $\eps$ by replacing the sharp truncation by a smooth truncation and using a dyadic decomposition along the critical line.
%\end{rmk}

We will later require asymptotics for the following modification of this average. For $t>0$, let
\begin{align}\label{Phi}
\Phi_\fab(t) \coloneqq t \sum_{\substack{c\in\cN_\fab^*\\ c\geq t}} \frac{\varphi_\fab(c)}{c^3}.
\end{align}

\begin{coro}\label{asymptot}
Let $\delta=\min\{2(1-\rho_1),2/3\}\in(0,2/3]$ and $c_\Gamma$ defined as in \cref{c_Gamma}. Then
\[
\Phi_\fab(t) = \begin{dcases}  c_\G  +O\left(t^{-\delta}\right) & \text{ as }t\to\infty,\\
O(t) & \text{ as }t\to0.
\end{dcases}
\]
\end{coro}

\begin{proof}
The first statement follows by Abel summation on \cref{Good}; indeed, we have for any $T\gg t$,
\[
\sum_{\substack{c\in\cN_\fab\\ t< c\leq T}} \frac{\varphi_\fab(c)}{c^3} = \sum_{0<c\leq T} \varphi_\fab(c) T^{-3} - \sum_{0<c\leq t} \varphi_\fab(c) t^{-3} + 3\int_t^T \sum_{0<c\leq u} \varphi_\fab(c) u^{-4}du.
\]
Letting $T\to\infty$, 
\[
\sum_{\substack{c\in\cN_\fab\\ t\leq c}} \frac{\varphi_\fab(c)}{c^3} = c_\G t^{-1} %+ \sum_{j=1}^k c_j t^{2\rho_j -3}\left(\frac{1}{3-2\rho_j}-1\right) 
+ O\left(t^{-\delta}\right). 
\]
%Multiplying by $t$ and comparing the exponents, we get the desired result where we by replacing the constants $c_j$ with $c_j \left(\frac{1}{3-2\rho_j}-1\right)$ .
For the second statement, we use \cref{Good} and positivity to see $\varphi_\fab(c) \ll c^{2-\eps}$ for some $\eps>0$. Hence the sum defining $\Phi_\fab(t)$ is absolutely convergent.
\end{proof}

%Previous version:
%Section 8: A mean value formula for pairs of discrete lattice orbits
%Section 9: A second mean value formula

\section{Mean values formulas for pairs}\label{sec:2ndmoment1}
A formula for the second moment of the Siegel--Veech transform is derived in \cite{Fai21} when the lattice $\G=H_q$ is a Hecke triangle group. 
In this setting there is only one cusp class, represented by the cusp at $
\infty$ and we have $\Lambda:=\Lambda_\infty=\G\be_1$. Then for $f\in SC(\R^2\times\R^2)$ we have 
\begin{align*}
\int_{G/H_q} \Theta_{\Lambda,\Lambda;f}(g) d\mu(g) &=  c_{H_q}\sum_{c\in\mathcal{N}^*} \varphi(c) \int_G f(g\be_1,c g \be_2)d\eta(g)\\
& \qquad + c_{H_q} \int_{\R^2} (f(\bx,-\bx)+f(\bx,\bx))d\bx
\end{align*}
where $\cN=\cN_{\infty\infty}$, $\varphi=\varphi_{\infty\infty}$, %$J_c=\bsm 1&0\\0&c\esm$, 
and $\eta$ is the Haar measure on $G$ normalized so that $\eta_\ast(G/H_q)=c_{H_q}$ \cite[Theorem 1.2]{Fai21}. 

Opening up the measure $d\eta=da\,db\, ds$ under the coordinates in \cref{sec:Notation}, the inner integral becomes 
\begin{align*}
\iiint_{\substack{a,b,s\in \R \\ a \neq 0}}  f\left(\bsm a\\ as\esm, c\bsm b\\ bs+a^{-1}\esm\right) da\,db\, ds &=
%\iint  f(\bsm u\\ v\esm, t\bsm b\\ (bv+1)u^{-1}\esm) dt\frac{db}{|u|}dv du\\
%& = 
\frac{1}{|c|} \iiint_{\substack{u,v,x\in \R \\ u \neq 0}}  f\left(\bsm u\\ v\esm, x\bsm u\\ v\esm + c\bsm 0  \\ u^{-1}\esm\right)dx\,dv\,du\\
&=\frac{1}{|c|}\int_\R\int_{\R^2} f(\bx,t\bx+c\bx^*)d\bx\, dt.
\end{align*}
under the change of coordinates $a=u$, $s=vu^{-1}$, $ux=c b$ which has Jacobian $1/|c|$, and for $\bx^*$ chosen such that $\det(\bx\, |\, \bx^*)=1$.
%The change of variables from the coordinates $(u,v,x,c)$ to $(\bx, \by)$ has Jacobian $1$, and this proves (\ref{wts}). 

We can extend this identity from Hecke triangle groups to all non-uniform lattices and pairs of not necessarily homothetic discrete lattice orbits. The mean value formula in \cref{Thm1} follows from

\begin{thm}\label{2ndMoment}
Given two scaled discrete $\G$-orbits $\Lambda_\fa, \Lambda_\fb$, $f \in SC(\R^2\times \R^2)$,  we have 
\begin{align*}
\int_{G/\G} \Theta_{\Lambda_\fa,\Lambda_\fb;f}(g)d\mu(g)& = c_\Gamma \sum_{c\in\cN_\fab^*} \varphi_\fab(c) \int_G f(g\be_1,cg\be_2)d\eta(g)\\
&\qquad + \delta_\fab c_\Gamma \int_{\R^2} (f(\bx,-\bx)+f(\bx,\bx))d\bx,
\end{align*}
where $\delta_\fab=1$ if and only if $\Lambda_\fa$ and $\Lambda_\fb$ are homothetic.
\end{thm}
\begin{rmk}
 	In order to obtain an extension from $SC(\R^2\times \R^2)$ to a larger class of functions, one can follow the techniques of \cref{Veech}, where the co-area formula is replaced by the area formula. This technique does not unconditionally yield the set of all integrable Borel functions. 	
\end{rmk}
\begin{proof}%[Proof of \cref{2ndMoment}]
Following the discussion in \cite[Section 3.1]{Fai21}, we decompose $\R^2\times \R^2$ into $G$-invariant subsets on which $\mu$ must be supported; we write 
\begin{align}\label{decompo}
\int_{G/\G}\Theta_{\Lambda_\fa, \Lambda_\fb; f}(g) d\mu(g) = a\cdot\delta_0(f) + b_\infty\int_{\R^2} f(\bo,\bx)d\bx +\sum_{t\in \R} b_t\int_{\R^2} f(\bx,t\bx)d\bx\nonumber\\
+ \sum_{t\neq 0}c_t \int_G f(g\be_1, tg\be_2) d\eta(g).
\end{align}

In the following we use the fact that we can choose any $f\in SC(\R^2\times \R^2)$. In each case we will define $f = \one_S$ for a specific closed set $S$, and use it to determine the coefficients for each term of \cref{decompo}.

%$a=0$: 
Set $S = \{(\bo,\bo)\}$. Then since no lattice orbit passes through the origin, we have $a=0$.

%$b_0=b_\infty=0$: 
Set $S = \{\bo\}\times \overline{B_R}$. Since no discrete lattice orbit contains the origin, $\Theta_{\Lambda_\fa, \Lambda_\fb; f}=0$, and we have $b_\infty=0$. The same argument yields $b_0=0$.

%$\delta_\fab=0\implies b_t=0$ ($t\in\R^*)$: 
Set $S =\{(\bx,t\bx)\in \overline{B_R}\times \overline{B_R}\}$. By \cref{determinants and cosets}, if $\delta_\fab=0$, then $S\cap (\Lambda_\fa \times\Lambda_\fb)= \emptyset$. Then (\ref{decompo}) becomes
\[
0 = b_t \int_{\R^2} \one_S(\bx,t\bx) d\bx = b_t |B_{R/t}|
\]
and hence $b_t=0$ for all $t\in\R^*$ when $\delta_\fab=0$.
%\item %$\delta_\fab=1\implies t\in\{\pm1\}$ and $b_1=b_{-1}=c_\Gamma$: 

Suppose now $\delta_\fab=1$. 
Let $\bx\in\Lambda_\fa$, $t\bx\in\Lambda_\fa$. Choose $g\in\G\gs_\fa$ and $h\in\G\gs_\fa$ such that the first column vectors of $g$ and $h$ are respectively $\bx$ and $t\bx$. Then 
\[
h^{-1}g = \bpm 1/t &*\\ 0&t\epm \in\gs_\fa^{-1}\G\gs_\fa.
\] 
By \cref{Miyake}, we conclude that $t\in\{\pm1\}$. Set $f_+ = \one_S$ when $S =\{(\bx,\bx)\in \overline{B_R}\times \overline{B_R}\}$ and $f_- = \one_S$ when $S =\{(\bx,-\bx)\in \overline{B_R}\times \overline{B_R}\}$. Hence $\Theta_{\Lambda_\fa, \Lambda_\fb; f_{\pm}}\neq 0$ and (\ref{decompo}) yields
\[
\int_{G/\G}\Theta_{\Lambda_\fa, \Lambda_\fb; f_{\pm}} (g)d\mu(g) = b_{\pm1} \int_{\R^2} f_\pm(\bx,\pm \bx)d\bx.
\]
On the other hand, by \cref{Veech}, we have
\[
\int_{G/\G} \Theta_{\Lambda_\fa, \Lambda_\fb; f_{\pm}}(g)d\mu(g) = c_\Gamma \int_{\R^2} f_\pm(\bx,\pm\bx)d\bx.
\]
Since this is true for all $R$ we conclude that when $\delta_\fab=1$ we have $b_t=0$ when $t\neq\pm 1$ and $b_{\pm 1} = c_\Gamma$. 

%$c_t\in\cN_\fab^*$ and $c_t=c_0\cdot \varphi_\fab(t)$: 
Set $S=\{(\bx,\by) : \bx, \by \in \overline{B_R},\, \det( \bx \mid \by)=c\}$. Then $\Theta_{\Lambda_\fa, \Lambda_\fb; f}(g)=|\{(\bx,\by)\in(\Lambda_\fa\times\Lambda_\fb) \cap g^{-1} \overline{B_R}: \det(\bx \by)=c\}|\neq0$ if and only if $c\in\cN_\fab^*$. Now suppose $c\in \mathcal{N}_\fab^*$. By \cref{lemma2}, there exists $0<a\leq |c|$ so that $\begin{pmatrix} a\\ c\end{pmatrix} \in \sigma_{\mathfrak{a}}^{-1} \Gamma \sigma_{\mathfrak{b}},$ so that by definition $\varphi_{\fab}(c) \neq 0$.

We want to show that (cf. \cite[Lemma 3.6]{Fai21}) for all $R$, the support of $\Theta_{\Lambda_\fa, \Lambda_\fb; f}$ reduces to $\varphi_{\fab}(c)$ total $\Gamma$-orbits. More precisely we show that
\[D_c^{\fab} \defeq \{(\bx_1,\bx_2) \in \Lambda_\fa \times \Lambda_\fb: \det(\bx_1| \bx_2)=c\} = \bigsqcup_{\stackrel{\bsm a\\ c\esm \in \sigma_\fa^{-1} \Gamma \sigma_\fb\be_1 }{0<a\leq |c|}} 
\Gamma \sigma_\fa \begin{pmatrix} 1 & a \\ 0 & c\end{pmatrix}.\]

Indeed by the proof of \cref{lemma2}, there is some $h\in 
\Gamma\sigma_{\fa}$ so that $h^{-1}(\bx_1| \bx_2) = \bsm 1 & a \\ 0 & c\esm$ for some $0< a\leq |c|$ and $\bsm a\\ c\esm \in \sigma_\fa^{-1} \Gamma \sigma_\fb\be_1$. Moreover each of these orbits is distinct. Namely if $|c|\geq a_2  > a_1 >0$ with 
\[\Gamma \sigma_\fa \begin{pmatrix} 1 & a_1 \\ 0 & c\end{pmatrix} = \Gamma \sigma_\fa \begin{pmatrix} 1 & a_2 \\ 0 & c\end{pmatrix},\]
then since $\sigma_{\fa}$ is a scaling transformation, we must have $a_2-a_1 \in c\Z$ which is impossible as $0<a_2-a_1 <c.$ 

Now that we know $D_c^{\fab}$ is decomposed into $\varphi_{\fab}(|c|)$ distinct orbits, we want to show we have the same contribution for each orbit. Namely  (cf. \cite[Lemma 3.7]{Fai21}) if we fix $\bsm a\\ c\esm \in \sigma_{\mathfrak{a}}^{-1} \Gamma \sigma_{\mathfrak{b}}\be_1,$ with $0<a\leq |c|$, then by our normalization of the Haar measure, we have 
\begin{align*}
\int_{G/\G} \sum_{\gamma \in \Gamma} f\left(g \gamma \sigma_\fa \begin{pmatrix} 1 & a \\ 0 & c\end{pmatrix}\right) \,d\mu(g) &=c_\Gamma \int_G  f\left(g  \sigma_\fa \begin{pmatrix} 1 & a \\ 0 & c\end{pmatrix}\right)\,d\eta(g)\\
&= c_\Gamma \int_G f\left(g\bpm 1 &0\\ 0&c\epm\right)d\eta(g)
\end{align*}
by multiplying $g$ on the right by $\begin{bmatrix} 1 & -\frac{a}{c}\\ 0 & 1\end{bmatrix}\sigma_\fa^{-1}$ and using that $\eta$ is invariant under the action of $G$. Hence we get the same contribution from each of the $\varphi_\fab(|c|)$ orbits. Since this is true for all $R$ we obtain the desired coefficients, thus concluding the proof of the first equality.
\end{proof}

Let $C$ be the cone of $2\times2$ matrices $A$ such that $\lambda A\in G/\G$ for some $\lambda\geq1$, and let $dA$ be the Euclidean volume element. There exists some constant $c_0$ so that
\[
\int_C f(A) dA = c_0\int_0^1 \nu \int_{G/\G} f(\nu^{1/2}g) d\mu(g) d\nu.
\]
Inspired by Schmidt \cite{Schmidt}, we introduce the equivalent\footnote{ I.e., the measures $dm(A)$ and $dA$ have the same null measure sets.} cone measure 
\[
\int_C f(A) dm(A) = \int_0^1 \int_{G/\G} f(\nu^{1/2} g)d\mu(g) d\nu.
\]
Let $\Lambda_\fa$, $\Lambda_\fb$ be scaled discrete $\G$-orbits. With respect to the cone measure $dm$, we have $m(C)=1$, 
\begin{align}\label{F1_Lemma}
\int_C \Theta_{\Lambda_\fa;f}(A) \det(A)\, dm(A) &= \int_0^1 \nu \int_{G/\G} \Theta_{\Lambda_\fa;f}(\nu^{1/2}g)\, d\mu(g)\, d\nu = c_\Gamma \int_{\R^2} f(\bx)\, d\bx
\end{align}
and
\begin{align*}
\int_C \Theta_{\Lambda_\fa,\Lambda_\fb;f}(A) \det(A)^2\, dm(A) &= \int_0^1 \nu^2 \int_{G/\G} \Theta_{\Lambda_\fa,\Lambda_\fb;f}(\nu^{1/2}g)\, d\mu(g)\,d\nu.
\end{align*}
%\begin{rmk}
%	 Note the different normalizations of the cone measure are all equivalent in the sense that the Radon--Nikodym derivative with respect to $dA$ is always positive, so the measures agree on which sets have measure zero.
%\end{rmk}
The second mean value formula given by Theorem 1.8' follows from the following theorem.

\begin{thm}\label{thm-useful}
Let $\G<G$ be a nonuniform lattice containing $-I$, and let $\Lambda_\fa$, $\Lambda_\fb$ be two scaled discrete $\G$-orbits. For each $f\in SC(\R^2\times\R^2)$, we have
\begin{align}\label{MVF}
\int_C \Theta_{\Lambda_\fa,\Lambda_\fb;f}(A)\det(A)^2 dm(A) &= c_\Gamma\iint_{\R^2\times\R^2} \Phi_\fab(|\bx\wedge\by|)f(\bx,\by)d\bx d\by\nonumber\\
&\qquad  + \frac{{\delta_\fab}{c_\Gamma}}{2}\int_{\R^2}\left( f(\bx,\bx)+f(\bx,-\bx)\right) d\bx,
\end{align}
where $\Phi_\fab$ is the function given by (\ref{Phi}).
\end{thm}

\begin{proof}
We will obtain this formula from \cref{2ndMoment} by a change of coordinates. Define $f_{\nu}(\bx,\by) = f(\nu^{1/2} \bx, \nu^{1/2}\by)$, and apply \cref{2ndMoment} to $\Theta_{\Lambda_\fa,\Lambda_\fb;f_{\nu}}$. In the linearly dependent subsets, we have via subsitution
\[\delta_{\fab}c_\Gamma \int_0^1 \nu^2 \int_{\R^2} f_\nu(\bx,\bx)+ f_\nu(\bx , -\bx)\,d\bx\,d\nu = \delta_{\fab}c_\Gamma \int_0^1 \nu \,d\nu \int_{\R^2} f(\bx,\bx) + f(\bx,-\bx)\,d\bx. \]
 
It remains to show that
\begin{align}\label{wts}
c_\Gamma \int_0^1 
\nu^2 \sum_{c\in\cN_\fab^*} \varphi_\fab(c)  \int_G f(\nu^{1/2}gJ_c) d\eta(g)d\nu  & = c_\Gamma \iint \Phi_\fab(|\bx\wedge\by|)f(\bx,\by)d\bx d\by.
\end{align}
where we set $J_c=\bsm 1&0\\0&c\esm$. 
%By a simple change of coordinates $g=g'\mathrm{diag}(\nu^{-1/2}, \nu^{1/2})$ and the $G$-invariance of $\eta$
%\begin{align*}
%\int_G f(\nu^{1/2} g J_c)d\eta(g)= \int_G f(g' \nu^{1/2} h_{\nu,c} J_c)d\eta(g') = \int_G f(g\bsm 1&0\\0&c\nu\esm)d\eta(g).
%\end{align*}
We write $t=\nu c$ and note that 
\begin{enumerate}
\item $\cN_\fab$ and $\varphi_{\fab}(c)$ are symmetric about $0$ if $-I \in \Gamma$;
\item and $\cN_\fab$ is unbounded, so the parameter $t$ ranges over $\R\setminus\{0\}$.
\end{enumerate}
Since $\Theta_{\Lambda_\fa,\Lambda_\fb; f}$ is integrable, Fubini applies and we may exchange the order of integration and summation in what follows. 
 The LHS of (\ref{wts}) becomes 
\begin{align*}
 &c_\Gamma \sum_{c\in\cN_\fab^*} \varphi_\fab(c) \int_0^1 \nu^2 \int_G f(\nu^{1/2}gJ_c) d\eta(g)d\nu \tag{{By a simple change of coordinates $g=g'\mathrm{diag}(\nu^{-1/2}, \nu^{1/2})$ and the $G$-invariance of $\eta$}}\\
 &=c_\G\int_G \sum_{c\in\cN_\fab^*} \int_0^1 \nu^2 \varphi_\fab(c) f(g\bsm 1&0\\ 0& c\nu\esm) d\nu d\eta(g) \\
 &\stackrel{(1)}{=} c_\G \int_G \sum_{c\in (\cN_\fab^*)_{>0}} \left[\int_{0}^c t^2 \frac{\varphi_\fab(c)}{c^3} f(g\bsm 1&0\\0&t\esm) + \int_{-c}^0 t^2 \frac{\varphi_\fab(c)}{c^3}f(g\bsm 1&0\\0&t\esm)\right]\,dt\,d\eta(g)  \\
&\stackrel{(2)}{=}c_\G \int_G \left[\int_0^
\infty t^2f(g\bsm 1&0\\0&t\esm)  \sum_{\substack{c\in\cN_\fab^*\\ c\geq t}} \frac{\varphi_\fab(c)}{c^3} dt  + \int_{-\infty}^0 t^2 f(g\bsm 1&0\\0&t\esm)  \sum_{\substack{c\in\cN_\fab^*\\ -c\leq t}} \frac{\varphi_\fab(c)}{c^3} \right] d\eta(g)\\
&\stackrel{(1)}{=} c_\G\int_\R |t|\Phi_\fab(|t|) \int_G f(g\bsm 1&0\\ 0&t\esm) d\eta(g) dt.
\end{align*}
\end{proof}

\section{Geometric estimates}\label{sec:estimates}

We may rewrite the mean value formula as
\begin{equation}\label{eq:split}
\begin{split}
\int_C \Theta_{\Lambda_\fa,\Lambda_\fb;f}(A)\det(A)^2 dm(A) &= 
c_\Gamma^2 \iint_{\R^2\times\R^2} f(\bx,\by)d\bx d\by +\frac{\delta_\fab c_\Gamma}{2} \int_{\R^2}\op{f(\bx,\bx)+f(\bx,-\bx)}d\bx\\
&\qquad +
c_\Gamma \iint_{\R^2\times\R^2} \op{\Phi_\fab(|\bx\wedge\by|) - c_\Gamma}f(\bx,\by)\, d\bx\, d\by.
\end{split}
\end{equation}
In this section, we use integration by parts to compute estimates for the last double integral,
\begin{equation}
\label{2}
 \iint_{\R^2\times\R^2} \op{\Phi_\fab(|\bx\wedge\by|) - c_\Gamma}f(\bx,\by)\, d\bx\, d\by.
\end{equation} 
Namely, we will obtain an estimate in \eqref{pause}, and the results from this section to be used later are given in \cref{ex:BR}, \cref{ex:eps}, and \cref{ex:det}.

For $t,u>0$ consider the rectangle,
\[
P_{t,u}(\bx) = \{\by\in\R^2: |\bx^\perp\wedge \by|\leq t,\, |\bx\wedge\by|\leq u\}
\]
and set $P(\bx)=P_{1,1}(\bx)$. Then $tP(\bx)=P_{t,t}(\bx)$. If we write $\bx= \|\bx\| k_\theta\be_1$ and $\bx^\perp=k_{\pi/2} \|\bx\|k_\theta\be_1 =  \|\bx\|k_\theta \be_2$, we have 
\[
P_{t,u}(\bx) = k_\theta P_{\tfrac{t}{\|\bx\|},\tfrac{u}{\|\bx\|}}(\be_1).
\]
In particular, $P_{t,u}(\bx)$ is a rotated rectangle, centered at the origin, with sides of length $2t/\|\bx\|$ and $2u/\|\bx\|$. We consider the inner integral of (\ref{2}). For $t$ sufficiently large, writing $\bx=\|\bx\| k_{\theta} \be_1$, we have
\begin{align*}
\int_{\R^2} (\Phi_\fab(|\bx\wedge\by|)-c_\G &\delta_\fab) f(\bx,\by)d\by  = \int_{tP(\bx)} (\Phi_\fab(|\be_1\wedge \|\bx\|k_{-\theta}\by|)-c_\G) f(\bx,\by)\, d\by \\ 
&=\frac{1}{\norm{\bx}^2}\int_{tP(\be_1)} (\Phi_\fab(|\be_1\wedge \by|)-c_\G) f(\bx,\tfrac{1}{\|\bx\|}k_\theta \by)\, d\by \\ 
&= \frac{1}{\|\bx\|^2}\int_0^t (\Phi_\fab(y_2)-c_\G)\left( \int_0^{t} \sum f\left(\bx, \tfrac{1}{\|\bx\|} k_\theta \bsm \pm y_1\\ \pm y_2\esm\right)\, dy_1\right) dy_2,
\end{align*}
where the inner sum contains a summand for each of the four $\bsm \pm y_1\\ \pm y_2\esm$. We call the inner integral $g_t(y_2)$ and compute that its primitive is
\begin{align*}
G_t(u) &= \int_0^u \int_0^t \sum f\left(\bx, \tfrac{1}{\|\bx\|} k_\theta \bsm \pm r\\ \pm s\esm\right)\, dr\, ds\\
& = \int_{P_{t,u}(\be_1)} f(\bx,\tfrac{1}{\|\bx\|}k_\theta \by)\, d\by\\
& = \|\bx\|^2 \int_{P_{t,u}(\bx)} f(\bx,\by)\, d\by.
\end{align*}
Recall that \cref{asymptot}, there is some $t_0 >0$ and constant $M$ so that
\begin{equation}\label{eq:t0Phi}
\Phi_\fab(t) \leq \tilde{\Phi}_\fab(t) \defeq \begin{dcases}  c_\G +M t^{-\delta} & t \geq t_0,\\
M t & 0\leq t < t_0.
\end{dcases}
\end{equation}
Applying integration by parts yields
\begin{align}
\int_{\R^2} (\Phi_\fab(|\bx\wedge\by|)&-c_\G) f(\bx,\by)d\by \nonumber\\
&= \frac{1}{\|\bx\|^2}\int_0^t (\Phi_\fab(y_2)-c_\G)\left( \int_0^{t} \sum f\left(\bx, \tfrac{1}{\|\bx\|} k_\theta \bsm \pm y_1\\ \pm y_2\esm\right)\, dy_1\right) dy_2\nonumber \\ 
& \ll  (\tilde{\Phi}_\fab(t)-c_\G ) \int_{tP(\bx)} f(\bx,\by)\, d\by -\int_0^t \tilde{\Phi}'_\fab(u) \int_{P_{t,u}(\bx)} f(\bx,\by)\, d\by du\nonumber \\ 
& \ll t^{-\delta} \int_{tP(\bx)} f(\bx,\by)\, d\by+ \delta \int_{t_0}^{\infty} \left(\int_{P_{t,u}(\bx)} f(\bx,\by)d\by\right) u^{-1-\delta}du. \label{pause}
\end{align}

\begin{example}\label{ex:BR} If $f(\bx,\by)=1_{B_R}(\bx)1_{B_R}(\by)$. Letting $t\to\infty$ in (\ref{pause}), we have
\[
(\ref{2}) \ll \int_{t_0}^\infty \int_{B_R} |P_{\infty,u}(\bx)\cap B_R| d\bx\, u^{-1-\delta} du
\]
with $|P_{\infty,u}(\bx)\cap B_R| =|B_R|$ if $u\geq R\|\bx\|$, and $|P_{\infty,u}(\bx)\cap B_R|\ll \tfrac{Ru}{\|\bx\|}$ otherwise.
Hence
\begin{align*}
(\ref{2}) & 
\ll \int_{B_R}\left( \int_0^{R\|\bx\|} \frac{uR}{\|\bx\|} u^{-1-\delta}du +\int_{R\|\bx\|}^\infty R^2 u^{-1-\delta}du\right) d\bx%\\
%&\ll \int_{B_R}\left( \frac{R}{\|\bx\|} \left. \frac{u^{1-\delta}}{1-\delta}\right|^{R\|\bx\|}_0 + R^2 \left. \frac{u^{-\delta}}{-\delta}\right\vert_{R\|\bx\|}^\infty\right) d\bx\\
%& \ll R^{2-\delta} \int_{B_R} \frac{d\bx}{\|\bx\|^\delta} \left(\frac{1}{1-\delta}+\frac{1}{\delta}\right) 
\ll R^{2(2-\delta)}
\end{align*}

\end{example}

\begin{example}\label{ex:eps} Let $f(\bx,\by)=1_{B_R}(\bx)1_{B^*_\eps(\bx)}(\by)$. Once again, $|P_{\infty,u}(\bx)\cap B^*_\eps(\bx)| = |B_\eps|$ if $u\geq\eps\|\bx\|$ and $|P_{\infty,u}(\bx)\cap B^*_\eps(\bx)|\ll \tfrac{\eps u}{\|\bx\|}$ otherwise. Hence we have the estimate
\begin{align*}
(\ref{2}) & 
\ll \int_{B_R} \op{\int_0^{\eps\|\bx\|} \frac{\eps}{\|\bx\|} u^{-\delta} du + \int_{\eps\|\bx\|}^\infty \eps^2 u^{-\delta}\frac{du}{u}} d\bx \ll (\eps R)^{2-\delta}.
\end{align*}
\end{example}

\begin{proof}[Proof of \cref{thm:pairco}]
Apply \cref{eq:split} to the characteristic function $f$ supported on $\{\bx\in\R^2,\, \by\in B_{s/\sqrt c}^*(\bx)\}$ and use the estimate from \cref{ex:eps} to bound the remaining term.
\end{proof}

\begin{example}\label{ex:det} For $D,s >0$, let $f(\bx,\by)=1_{B_R}(\bx)1_{\mathcal{D}_{D,s}(\bx)}(\by)$, where
\[\mathcal{D}_{D,s}(\bx) = \{\by \in \R^2: |\by| \leq s |\bx|\text{ and } |\bx \wedge \by| \leq D\}.\] Then in this case $|P_{\infty, u}(\bx) \cap \mathcal{D}_{D,s}(\bx)| \leq \frac{u}{\norm{\bx}} \cdot s\norm{\bx} = s u$ if $u \leq D$, and $|P_{\infty, u}(\bx) \cap \mathcal{D}_{D,s}(\bx)| \leq s D$ if $u \geq D$. Therefore 
\[(\ref{2})
\ll s \int_{B_R} \op{\int_0^{D} u^{-\delta} du + D \int_{D}^\infty u^{-1-\delta} \,du}  d\bx \ll s D R^2.\]
\end{example}

\section{Effective counting}\label{sec10}

\subsection{Second moment estimate}

We will derive \cref{Thm:1} (see \cref{countN} for a full statement) by a combination of Borel--Cantelli and interpolation from the following `on average' power saving. Our strategy is inspired by \cite{Schmidt}.
\begin{prop}\label{prop:11}
	Let $\Lambda$ be a scaled discrete lattice orbit, and let $B$ be a Lebesgue measurable set in the plane with finite volume $|B|>C_\Lambda$. There is a constant $M_\Lambda$ so that
	\[
	%\int_C \left( \Theta_{\Lambda; \one_B}(A) - c_\Gamma |B||\det(A^{-1})|\right)^2|\det(A)|^2\,d m(A)= 
	\int_C \left(\det(A) \Theta_{\Lambda; \one_B}(A) - c_\Gamma |B|\right)^2\,d m(A) \leq M_\Lambda |B|^{2-\delta},\]
	where $\delta$ is given by \cref{def:delta}.
\end{prop}

To prove \cref{prop:11}, we use the following result of Schmidt for $n=2$.
\begin{lm}[\cite{Schmidt} Theorem 3]\label{lem:schmidthm3}
	Let $S$ be a Lebesgue mesurable set in $\R^2$ with characteristic function $\one_S$ and volume $|S|$. Then
	\[\iint \chi(|\bx \wedge \by|) \one_S(\bx) \one_S(\by)  \leq 8 |S| \int_0^\infty \chi(t) \,dt\]
	for every nonnegative, nonincreasing function $\chi(t)$ defined for $t\geq 0$ whose integral $\int_0^\infty \chi(t)\,dt$ converges.
\end{lm}
\begin{proof}[Proof of \cref{prop:11}]
	Let $\alpha(|B|)$ be a function increasing with $|B|$ that we will choose later. Choose $|B|$ large enough so that for $t_0 < \alpha(|B|)$ we can use \cref{asymptot} (cf. \cref{eq:t0Phi})  to bound
	\[|\Phi(t) - c_\Gamma| \leq \chi(t) +  \begin{cases} 0 & t\leq \alpha(|B|) \\
		M\alpha(|B|)^{-\delta} & t > \alpha(|B|) \end{cases}\]
		where for 
		\[\chi(t) = \begin{cases} M t_0  & 0 \leq  t< t_0 \\
		M t^{-\delta} &  t_0 \leq t < \alpha(|B|) \\ 
		0 & t \geq \alpha(|B|). \end{cases}\] 
	Notice $\chi(t)$ satisfies the assumptions of \cref{lem:schmidthm3} with $\int_0^\infty \chi(t)dt = M \left(t_0^2  - \frac{t_0^{1-\delta}}{1-\delta}\right) + \frac{M}{1-\delta} \alpha(|B|)^{1-\delta} .$
	Combining \cref{F1_Lemma}, \cref{eq:split}, and \cref{lem:schmidthm3} we compute
	\begin{align*}
		\int_C (\det(A) \Theta_{\Lambda; \one_B}(A) - c_\Gamma &|B|)^2\, d m(A)
		= \int_C \Theta_{\Lambda,\Lambda;\one_B}(A)\det(A)^2 dm(A) - c_\Gamma^2 |B|^2\\
		 &\leq  c_\Gamma |B| +
c_\Gamma \iint_{\R^2\times\R^2} \op{\Phi_\fab(|\bx\wedge\by|) - c_\Gamma}\one_B(\bx)\one_B(\by)\, d\bx\, d\by,\\
&\ll |B|(1 + \alpha(|B|)^{1-\delta}) + |B|^2 \alpha(|B|)^{-\delta}. 
	\end{align*}
	Choosing $\alpha(|B|) = |B|$, and noting that the implied constants only depend on $\Lambda$ completes the proof. 
\end{proof}

\subsection{Effective count via interpolation}

We can now derive the main result of this section. The proof follows the same lines as that of \cite[Theorem 2]{Schmidt} and \cite[Theorem 6.1]{KY21}. We include it for the convenience of the reader.

	\begin{thm}\label{countN}
Let $\G<G$ be a lattice containing $-I$ with scaled discrete orbit $\Lambda$. Let $\mathcal{B}$ be a linearly ordered family of Borel sets of finite volume in $\R^2$. Let $\psi$ be a positive non-increasing function so that $e^{t(2-\delta)} \psi(t)$ is eventually non-decreasing and $\int_1^\infty \psi(t)\,dt < \infty.$ Then for almost every linear transformation $A$
\[
|A(\Lambda)\cap B||\det(A)|= c_\G|B| + O\op{\frac{|B|^{1-\delta/2}\log^{1/2}(|B|)}{\psi^{1/2}(\log(|B|))}}
\]
\end{thm}

\begin{proof}
	Without loss of generality, we may assume $\mathcal{B}$ has sets of arbitrarily larges volumes. By \cite[Lemma 1]{Schmidt} we may assume without loss of generality that $\{|B|: B\in \mathcal{B}\} = \R^+$. Thus for all $N\in \mathbb{N}_{\geq 1}$ there exists $B_N \in \mathcal{B}$ so that $|B_N| = N$. 
	
	Define \[S_N(A) = |A\Lambda \cap B_N| - c_\Gamma N |\det(A^{-1})|\]
	and for $0\leq N_1 < N_2 $
	\[{}_{N_1}{S}_{N_2}(A) =  |A\Lambda \cap (B_{N_2}\setminus B_{N_1})| - c_\Gamma (N_2-N_1) |\det(A^{-1})|.\]
	
	For $T\geq 3$ set 
	\[\mathcal{K}_T = \left\{(N_1, N_2)\in \Z^2\, |\, 0\leq N_1<N_2\leq 2^T,\, N_1 = \ell 2^t \text{ and } N_2 = (\ell+1) 2^t \text{ for some } \ell, t \in \mathbb{N}_{\geq 0}\right\}.\]
	By \cref{prop:11} with $B= B_{N_2} \setminus B_{N_1}$, notice that each value of $N_2-N_1 = 2^t$ occurs $2^{T-t}$ times so we have 
	\begin{equation}\label{lem:12}
	\sum_{(N_1,N_2)\in \mathcal{K}_T}\int_C \left|{}_{N_1}{S}_{N_2}(A)\right|^2 |\det(A)|^2\,d m(A) \leq M_\Lambda \sum_{t=0}^T 2^{T - t} 2^{t(2-\delta)}\leq 8 M_\Lambda 2^{(2-\delta)T}, %= M_\Lambda 2^T \left(\frac{2^{(1-\delta)(T+1)} - 1}{2^{(1-\delta)} - 1}\right)
	\end{equation}
	where the last inequality follows from the geometric series and $\frac{1}{4} \leq 2^{1/3} \leq 2^{1-\delta} \leq 2$.
	
	Next let $\mathcal{E}_T \subseteq C$ be the set of all $A \in C$ so that
	\begin{equation}\label{eq:6.2}
		\sum_{(N_1,N_2)\in \mathcal{K}_T} |{}_{N_1}{S}_{N_2}(A)|^2 |\det(A)|^2 > \frac{2^{(2-\delta)T}}{48 \psi((T-1)\log(2))}.
	\end{equation}
	%Consider $d\tilde{m}(A) = |\det(A)|^2 \,dm(A)$ which is an equivalent measure to $dm$ which is equivalent to $dA$. 
	Applying Chebyshev's inequality and \eqref{lem:12} implies
	\begin{align}\label{eq:6.3}
		%\tilde{m}(E_T)
		 m(\mathcal{E}_T) &< \frac{48 \psi((T-1)\log(2))}{ 2^{(2-\delta)T}} \int_{\mathcal{E}_T} 
		\sum_{(N_1,N_2)\in \mathcal{K}_T} 
		|{}_{N_1}{S}_{N_2}(A)|^2 |\det(A)|^2\,d{m}(A) \nonumber\\
		& \leq 48\cdot 8 M_\Lambda \psi((T-1) \log(2)).
	\end{align}
	By the Borel--Cantelli lemma and integrability of $\psi$, $m(\mathcal{E}_\infty) = 0$ for $\mathcal{E}_\infty = \limsup \mathcal{E}_T$. Define $C\setminus \mathcal{E}_\infty$ to be the full measure set (with respect to $\,dA$) of transformations where the discrepancy is small.
	
	For $N\leq 2^T$, we can write $[0,N) \subseteq \bigsqcup_{(N_1,N_2) \in \mathcal{I}} [N_1,N_2)$ where $\mathcal{I} \subset \mathcal{K}_T$ is given by 
		\[\mathcal{I} = \left\{(0\cdot 2^1 , 1\cdot 2^1)\right\} \cup\left\{ ( 2^{t}, 2\cdot 2^{t})\right\}_{t=1}^{T-1}.\] Thus we write
	\[S_N(A) \leq \sum_{(N_1,N_2) \in \mathcal{I}} {}_{N_1}{S}_{N_2}(A).\]
By (\ref{eq:6.2}), for any $A \notin \mathcal{E}_T$ and any $N < 2^T$,
	\begin{align*}
	|S_N(A)|^2 |\det(A)|^2 &\leq \left|\sum_{(N_1,N_2) \in \mathcal{I}} {}_{N_1}{S}_{N_2}(A)\right|^2 |\det(A)|^2\\
	&  \leq |\mathcal{I}| \sum_{(N_1,N_2) \in \mathcal{K}_T} \left| {}_{N_1}{S}_{N_2}(A)\right|^2 |\det(A)|^2\\
	& \leq T \frac{2^{(2-\delta)T}}{48 \psi((T-1)\log(2))} .
	\end{align*}
Consider $A \in C\setminus  \mathcal{E}_\infty$. Choose $T_A \in \mathbb{N}$ so that for all $T\geq T_A$, $A \notin \mathcal{E}_T$. That is for all $T \geq T_A$ and for any $N < 2^T$, $|S_N(A)|^2 |\det(A)|^2  \leq \frac{T 2^{(2-\delta)T}}{48 \psi((T-1)\log(2))}.$

In order to interpolate from $B_N$ to any $B$, we will set $N_A = \max\{2^{T_A}, N_0\}$ where the assumption that $e^{t(2-\delta)} \psi(t)$ is eventually non-decreasing ensures there exists some $N_0$ so that for all $N\geq N_0$,
\begin{equation}\label{eq:inter}
	\frac{(N+1)^{(1-\frac{\delta}{2})}\sqrt{\log(N+1)}}{2\psi^\frac{1}{2}(\log(N+1))} + 1 \leq \frac{N^{(1-\frac{\delta}{2})}\sqrt{\log(N)}}{\psi^\frac{1}{2}(\log(N))}.
\end{equation}

 Whenever $N > N_A$, we have $N > 2^{T_A}$ so we can choose an integer $T \geq T_A$ so that $2^{T-1} \leq N \leq 2^T$. For such $N$ we have
\[|S_N(A)|^2 |\det(A)|^2 \leq \left(\frac{\log(N)}{\log(2)} + 1\right) \frac{4 N^{2-\delta}}{48\psi(\log(N))} \leq  \frac{N^{2-\delta} \log(N)}{4\psi(\log(N))}\]
where we used $\psi$ in non-increasing, $2^{(2-\delta)T} = 4 \cdot 2^{2(T-1)} 2^{-\delta T}$, and $\left(\frac{\log(N)}{\log(2)} + 1\right) \leq 3 \log(N)$ for all $N\geq 2$.

We have now shown for all $N > 2^{T_A}$ that
\[|S_N(A)| |\det(A)|\leq  \frac{N^{(1-\frac{\delta}{2})} \sqrt{\log(N)}}{2\psi^\frac{1}{2}(\log(N))}\]

For any $B\in\mathcal{B}$ with $|B| \geq N_A$ there exists an integer $N \geq N_A-1$ so that $N \geq 2^{T_A}$ and $B_{N} \subseteq B \subseteq B_{N+1}$. We then interpolate and use \eqref{eq:inter} to obtain
\begin{align} \label{eq:countNbound}
\Big||A\Lambda \cap B| - c_\Gamma |B| |\det(A^{-1})|\Big| & \leq  \max\{|S_{N}(A)|, |S_{N+1}(A)|\} + \frac{c_\G}{|\det(A)|}\nonumber\\
& \leq   \frac{|B|^{(1-\frac{\delta}{2})}\sqrt{\log(|B|)}}{\psi^\frac{1}{2}(\log(|B|)) |\det(A)|}.\end{align}

Since every $\tilde{A}$ with $0<|\det(\tilde{A})|\leq 1$ is of the form $\tilde{A} = A \gamma$ for $A\in C$ and $\gamma \in \Gamma$, and since $\Gamma$ is enumerable, \eqref{eq:countNbound} holds for almost every linear transformation $\tilde{A}$ with $|\det(\tilde{A})|\leq 1$. Now consider $C$ a linear transformation with $\det(C) = c$. Applying \eqref{eq:countNbound} to the family of sets $\{C^{-1}B| B\in \mathcal{B}\}$, 
\[\Big||A\Lambda \cap C^{-1}B| - c_\Gamma |C^{-1}B| |\det(A^{-1})|\Big|=\Big||CA\Lambda \cap B| - c_\Gamma |B| |\det((CA)^{-1})|\Big|\]
the fact that $\det(C)$ can be taken arbitrarily large, we conclude \eqref{eq:countNbound} holds for almost every linear transformation.
%
%For the equivalent formulation, for almost every linear transformation consider $A^{-1}$ and we have
%\[\Big||\Lambda \cap A^{-1}B| - c_\Gamma |A^{-1}B|\Big| =\Big||A\Lambda \cap B| - c_\Gamma |B| |\det(A^{-1})|\Big| \leq    \frac{1}{|\det(A)|}\frac{|B|^{(1-\frac{\delta}{2})}\sqrt{\log(|B|)}}{\psi^\frac{1}{2}(\log(|B|))}.\]
\end{proof}

\subsection{Application to dilated sets} \label{sec:dilates}
We obtain the discrete lattice orbit version of \cref{Thm:1} by applying \cref{countN} to dilated sets.
\begin{thm}\label{Thm:1DLO}
Let $\Lambda$ be a scaled discrete $\G$-orbit and let $\Omega\subset\R^2$ be a bounded Borel set that contains the origin. Consider its dilates $\Omega_R=R\cdot \Omega$. Then for almost every linear transformation $A$ we have
\[
|A(\Lambda)\cap \Omega_R||\det(A)| = c_\G |\Omega| R^2 + O\left(R^{2-\delta}  \log^{3/2}(R)\right),
\]
where $\delta$ is as in (\ref{def:delta}).
\end{thm}

\begin{proof}
	Let $\psi(t) = t^{-2}$, which is integrable, non-increasing for $t > 1$, and satisfies that $e^{t(2-\delta)} \psi(t)$ is eventually non-decreasing. Now apply \cref{countN} to the family $\mathcal{B} = \{\Omega_R : R\in \R_{\geq 0}\}$ of dilates.
%   to obtain that for almost every linear transformation $A$,
% 	\[|\det(A)||A\Lambda \cap \Omega_R | = c_\Gamma |\Omega| R^2 + O_{|\Omega|, \delta}\left(R^{2-\delta}  \log(R) \log\log(R)\right).\]
\end{proof} 
 
We recall the following known estimates from \cite{BNRW}. For simplicity, suppose that $\G$ has trivial residual spectrum, i.e., $\delta = \frac{2}{3}$. By \cref{Thm:1DLO}, for \textit{almost every} linear transformation $A\in{\rm GL}_2(\R)$ we find
	\[|\det(A)||A(\Lambda) \cap \Omega_R | = c_\Gamma |\Omega| R^2 + O\left(R^{\frac{4}{3}}\log^{3/2}(R)\right)\] 
	\color{black}
while in \cite{BNRW} it is shown that for {\em every} $A\in G$ (hence with $\det(A)=1$), we have 
\[
|A(\Lambda)\cap \Omega_R| = c_\Gamma|\Omega|R^2 + O\left(R^\alpha\right),
\]
with
\[
\alpha =\begin{dcases} \frac43 & \text{ when }\Omega\text{ is the unit disk;}\\
\frac{8}{5} &\text{ when } \Omega\text{ is a star-shaped domain with smooth boundary;}\\
\frac{7}{4}+\eps & \text{ when }\Omega\text{ is a star-shaped domain with piecewise Lipschitz boundary;}\\
\frac{12}{7}+\eps & \text{ when }\Omega\text{ is a sector.}
\end{dcases}
\] 

%%%%%%%%%%%%%%%%%%%%%%%%%%%%%%%%%%%%%%%%%%%%%%%%%%
\section{Upper bounds on pair relationships}\label{sec:friendsbounds}
In this section, we obtain asymptotic upper bounds for the number of ordered pairs $(\bx,\by)\in\Lambda_1\times\Lambda_2$, where 
\begin{itemize}
\item $\Lambda_1$, $\Lambda_2$ are scaled discrete $\G$-orbits;
\item the first vector $\bx$ is restricted to $\bx \in B_R\cap \Lambda_1$;
\item the second vector $\by\in\Lambda_2$ has a prescribed metric relationship to $\bx$. 
\end{itemize}

The first relationship we examine is given by those vectors $\by \in \Lambda_2$ distinct from $\bx$ that are within distance $\epsilon$ of $\bx$. We say that two vectors $\bx,\by$ are {\bf $\eps$-friends} if $\|\bx-\by\|<\eps$. The second relationship examined is that of a determinant, first studied in \cite{AFM22}. A key aspect of these results is that they hold for every Veech surface. This will be used in Appendix~\ref{AppJon} to extend results previously only known for generic translation surface to every Veech surface.

\subsection{$\eps$-friends}

The main result of this section is

\begin{thm}\label{main}
Let $\Lambda_\fa$, $\Lambda_\fb$ be two scaled discrete $\G$-orbits. There exists a constant $C$ (depending only on $\G$) such that
\[
\limsup_{R\to\infty} \frac{|\{(\bx,\by)\in\Lambda_\fa\times\Lambda_\fb : \bx\in B_R,\, \by\in B^*_\eta(\bx)\}|}{|B_R|} < C|B_\eta|.
\]
\end{thm}

\begin{proof}
Let $f=\one_S$ be the characteristic function supported on 
\begin{align}\label{set}
S := S(R,\eta) = \{ (\bx,\by)\in\R^2\times\R^2: \bx\in B_R,\ \by\in B_{\eta}^*(\bx)\},
\end{align}
and let
\[
\Theta(A) := \Theta_{\Lambda_\fa,\Lambda_\fb;f}(A) = \sum_{\bx\in\Lambda_\fa}\sum_{\by\in\Lambda_\fb} f(A\bx,A\by) 
\]
for each $A\in\GL_2(\R)$. Fix $\eps>0$. We will show that we can choose $\eta>0$ such that
\begin{align*}%\label{claim}
\limsup_{R\to\infty} \frac{\Theta(I)}{|B_R|} < \eps.
\end{align*}
%whereby we recall that $|\Lambda\cap B_R| = c_\G|B_R| +O(R^{2-\delta})$. 

Fix $\alpha>0$ and consider the open symmetric neighborhood $\cU_\alpha =K\{ \bsm e^{r/2}&\\ & e^{-r/2}\esm: r\in(-\alpha,\alpha)\}K\subset G$ 
%$K\{a_y: y\in(e^{-\alpha},e^\alpha)\}K\subset G$ 
(in terms of the Cartan decomposition). If $f_\alpha$ denotes the characteristic function supported on $S(\sqrt{2}Re^{\alpha/2},\sqrt2 \eta e^{\alpha/2})$ and we set $\Theta_\alpha:=\Theta_{\Lambda_\fa,\Lambda_\fb;f_\alpha}$ then for each $A\in \cU_\alpha$, 
\begin{equation}\label{eq:alphafat}
\Theta(I)\leq \Theta_\alpha(A).
\end{equation}
Indeed, if $\bx\in B_R$, then $\|A\bx\|<\sqrt2 e^{\alpha/2}R$,
%$A\bx \in A B_R\subset B_{Re^\alpha}$ 
 and if $\by\in B_\eta^*(\bx)$, then 
$\|A\bx -A\by\| < \sqrt{2}\eta e^{\alpha/2}$.
%$\|A\bx-A\by\|\leq\|a_{y(A)}\|\|\bx-\by\|\leq  \eta e^\alpha$.

Let $C_\alpha$ be the matrix cone over $\cU_\alpha\G/\G$. %and let $\mathbf{1}_\alpha:C\to\R$ be the characteristic function supported on $C_\alpha$. 
Then the above estimate together with the fact that $\Theta_\alpha\geq0$ yield
\begin{align}\label{rhs}
\Theta(I) \frac{\mu(\cU_\alpha\G/\G)}{3} = \int_{C_\alpha} \Theta(I)\det(A)^2\, dm(A) \leq \int_{C} \Theta_\alpha(A) \det(A)^2 dm(A).
\end{align}
Expressing the Haar measure via the Cartan decomposition we find
\[
\mu(\cU_\alpha\G/\G) = \tfrac{1}{2\pi V}\int_{-\alpha}^\alpha |\sinh(r)|dr =\tfrac{1}{\pi V}(\cosh(\alpha)-1).\]
%Then the above inequality yields
%\begin{equation}\label{eq:alphaup}
%\begin{split}
%\int_C \Theta_{\Lambda_\fa,\Lambda_\fb;f}(I) \det(A)^2 dm(A)&  \leq \int_C \Theta_{\Lambda_\fa,\Lambda_\fb;f_\alpha}(A)\mathbf{1}_\alpha(A) \det(A)^2 dm(A)\\
%& \leq \int_C \Theta_{\Lambda_\fa,\Lambda_\fb;f_\alpha}(A) \det(A)^2 dm(A),
%\end{split}
%\end{equation}
%where the last inequality is trivial by positivity of the terms in the integrand. On the other hand, the LHS is equal to $\tfrac13 \Theta_{\Lambda_\fa,\Lambda_\fb;f}(I)$. 
By \eqref{eq:split} and \cref{ex:eps} the RHS of \eqref{rhs} is bounded above by
\begin{align*}
%\Theta(I) \frac{\mu(\cU_\alpha\G/\G)}{3} \leq
& c_\G^2 |B_{\sqrt2 R e^{\alpha/2}}||B_{\sqrt2 \eta e^{\alpha/2}}| + c_\G \iint_{\R^2\times\R^2} \op{\Phi_\fab(|\bx\wedge\by|)-c_\G}f_{\alpha}(\bx,\by)\, d\bx \, d\by\\
& = 4c_\G^2 e^{2\alpha} |B_{R}||B_{\eta}|  +O\op{(|B_{\sqrt2 Re^{\alpha/2}}||B_{\sqrt2 \eta e^{\alpha/2}}|)^{1-\delta/2}}.
\end{align*}
%On the other hand
%\[
%\mu(\cU_\alpha) = 2\pi \int_1^{e^\alpha} (y^2-y^{-2})\frac{dy}{y} = \frac12 \op{e^{2\alpha}+e^{-2\alpha}-2}.
%\]
%We conclude that there is a constant $c=c(\alpha)>0$ such that
%\begin{align*}
%\limsup_{R\to\infty} \frac{\Theta(I)}{|B_R|} & \leq\, c |B_\eta|
%\end{align*}
%%\limsup_{R\to\infty}  3\pi c_\G^2 \eta^2 e^{4\alpha} \left(1+ O((|B_{Re^\alpha}||B_{\eta e^\alpha}|)^{-\delta/2})\right) =  3\pi c_\G^2 \eta^2e^{4\alpha}.
%%\end{align*}
%and for any $\eps>0$ we may choose $\eta>0$ such that $c |B_\eta|<\eps$. 
Hence
\[
\limsup_{R\to\infty}\frac{\Theta(I)}{|B_R|} \leq 3 c_\Gamma |B_\eta| \frac{e^{2\alpha}}{\cosh(\alpha)-1}.
\]
The function in $\alpha$ has a local minimum at $\alpha=\ln(3)$ with value $\tfrac{27}{2}$.
\end{proof}

%
%In metric geometry, a discrete set $\Lambda\subset\R^2$ is said to be {\bf $\eps$-uniformly discrete} if for any $\bx\in\R^2$, $|\Lambda\cap B_\eps(\bx)|\leq 1$. In other words, a discrete set $\Lambda$ is $\eps$-uniformly discrete if the $\bx\in\Lambda$ with no $\eps$-friends have density 1. By \cref{main} discrete lattice orbits can be thought of as `weakly' uniformly discrete in the following sense.
%
%
%\begin{coro}\label{coro:quantdisc}
%Let $\Lambda_\fa$, $\Lambda_\fb$ be two scaled discrete $\G$-orbits. For any $\eps>0$, there is a $\eta>0$ such that
%\[
%\limsup_{R\to\infty} \frac{|\{\bx\in \Lambda_\fa\cap B_R : \bx \text{ has an }\eta\text{-friend in }\Lambda_\fb\}|}{|B_R|} \leq \eps.
%\]
%\end{coro}

The proof of \cref{main} extends to the set of holonomies of a Veech surface.

\begin{thm}
Let $M$ be a Veech surface. Then there exists a constant $C>0$, depending only on $M$, such that
\[
\limsup_{R\to\infty} \frac{|\{(\bx,\by)\in S_M\times S_M: \bx\in B_R,\, \by \in B^*_\eta(\bx)\}|}{|S_M\cap B_R|} <C|B_\eta|.
\]
\end{thm}

\begin{proof}
Recall that if $M$ is a Veech surface, its Veech group $\G_M$ is a nonuniform lattice in $G$ and the set $S_M$ is a disjoint finite union of discrete $\G_M$-orbits, i.e., $S_M=\cup \G_M \bx_i$. To each orbit corresponds a scaling factor $\lambda_i\neq0$ given by $\G_M\bx_i = \lambda_i \Lambda_i$ --- here $\Lambda_i$ is the associated scaled discrete orbit --- and we have
\[
|S_M\cap B_\eta^*(\bx)| = \sum |\Lambda_\fa \cap B^*_{\eta/\lambda_i}(\bx/\lambda_i)|.
\]
Counting the number of pairs that are $\eta$-friends we find
\[
\sum_{\bx\in S_M\cap B_R}\sum_{\by\in S_M\cap B_\eta^*(\bx)} 1 = 
\sum_{i,j}  \sum_{\bx\in \Lambda_j\cap B_{R/\lambda_j}} \left|\Lambda_i\cap B^*_{\eta/\lambda_i}\left(\tfrac{\lambda_j}{\lambda_i} \bx\right)\right|
\]
Let $f^{ij}$ be the characteristic function supported on the set
\[
\left\{(\bx,\by)\in \R^2\times\R^2: \bx\in B_{R/\lambda_j},\, \by\in B^*_{\eta/\lambda_i}\left(\tfrac{\lambda_j}{\lambda_i}\bx\right)\right\}.
\]
Running through the proof of \cref{main} we find that 
\[
\limsup_{R\to\infty}\, \sum_{i,j} \frac{\Theta_{\Lambda_j,\Lambda_i; f^{ij}_\alpha}(I)}{|S_M\cap B_R|} \ll_\G \sum_{i,j} \frac{c_\G^2}{\lambda_i^2 \lambda_j^2} \frac{|B_\eta|}{c_M} = c_M |B_\eta|.
\]
\end{proof}

\cref{thm:epsM} now follows immediately. We note that this result is trivially true for arithmetic lattice surfaces (e.g., square-tiled surfaces.) 

\begin{prop} \label{prop:arithuni}
If $M$ is an arithmetic lattice surface, then $S_M$ is uniformly discrete.
\end{prop}
\begin{proof}
Let $\Gamma_M$ be its arithmetic Veech lattice. As such, it is commensurable with $\SL_2(\Z)$, and there exists a lattice $L\subset \R^2$ such that any two points in $S_M$ differ by an element of $L$, see Lemma 5.4 and Theorem 5.5 in \cite{GJ}.  Let $\ell$ be the length of a shortest lattice vector in $L$. Then for any $\bx\in \R^2$, $|S_M\cap B_{\ell/2}(\bx)|\leq 1.$
\end{proof}

\subsection{Determinants} \label{subsec:det}

\begin{thm}\label{mainD}
 Fix $D>0$.  Let $\Lambda_\fa$ and $\Lambda_\fb$ be scaled discrete lattice orbits. There is a constant $C= C_\Gamma$ so that
\[
\limsup_{R\to\infty} \frac{|\{(\bx,\by)\in\Lambda_\fa\times\Lambda_\fb : \bx\in B_R,\, \by\in \mathcal{D}_{D,1}(\bx)|}{|B_R|} \leq   D C_\Gamma + 2 \delta_{\fab}c_\Gamma.
\]
\end{thm}

\begin{proof}
We follow the proof of \cref{main}, where $f_\alpha$ is the characteristic function supported on 
\begin{align*}
\{ (\bx,\by)\in\R^2\times\R^2: \bx\in B_{Re^\alpha},\ \by\in \mathcal{D}_{D,e^{2\alpha}}(\bx)\}.
\end{align*}
The conclusion comes from looking at \cref{ex:det}, and noticing that the linearly dependent terms do indeed contribute to the upper bound since for small enough determinant, we are counting parallel pairs which is determined by the Siegel--Veech constant.
\end{proof}

The extension to holonomy vectors for Veech surfaces given by \cref{thm:detM} is as straightforward as in the case of $\eps$-friends.

%
%We now apply the above result to prove the following for non-scaled discrete lattice orbits.
%\begin{prop}\label{prop:det}
%	Let $\Gamma \subset G$ be a lattice, and let $\Lambda_1, \Lambda_2$ be discrete $\Gamma$-orbits. For any $D>0$, there are constants $C_\G$ depending only on $\G$ and the scalings of $\Lambda_1$ and $\Lambda_2$ so that
%	\[
%\limsup_{R\to\infty} \frac{|\{(\bx, \by) \in \Lambda_1 \times \Lambda_2: \bx \in B_R,\,\, \by \in \mathcal{D}_{D,1}(\bx)\}|}{R^2} \leq C_\Gamma D +  2\delta c_\Gamma
%\]
%where $\delta$ is $1$ if  and only if $\Lambda_1$ and $\Lambda_2$ are homothetic.
%\end{prop}
%\begin{proof}[Proof of \cref{prop:det}]
%This directly follows the proof of \cref{prop:eps}.
%\end{proof}
%\begin{proof}[Proof of \cref{thm:detM}]
%Applying the same reasoning as the proof of \cref{thm:epsM}, the scaling is as follows
%\begin{align*}
%\sum_{\bx\in S_M\cap B_R}\sum_{\by\in S_M\cap \mathcal{D}_{D,1}(\bx)} 1 &=\, 
%\sum_{i,j} \sum_{\bx\in \Lambda_{j}\cap B_{\lambda_j R}} \sum_{y\in\Lambda_i\cap \mathcal{D}_{\lambda_i D, \lambda_i}(\tfrac{1}{\lambda_j}\bx)} 1 = \sum_{i,j} \sum_{\bx\in \Lambda_{j}\cap B_{\lambda_j R}} \sum_{y\in\Lambda_i\cap \mathcal{D}_{\tfrac{\lambda_i}{\lambda_j} D, \tfrac{\lambda_i}{\lambda_j}}(\bx)} .
%\end{align*}
%\end{proof}

\appendix
\section{An application to flows on translation surfaces} \label{AppJon}

 The first subsection introduces the notion of disjointness, and gives a criterion via spectral arguments for a family of flows to be disjoint. The second subsection verifies the criterion for non-arithmetic Veech surfaces. The main result of this section is Theorem~\ref{thm:disjoint trans} below. Theorem~\ref{thm:disjoint trans} implies Theorem~\ref{thm:disjoint} by \cref{thm:epsM}; see the proof of \cref{cor:redux}. The proof uses \cite{AD} and a largely geometric argument based on \cref{thm:epsM} and \cite{Vor}.

\subsection{A criterium for disjointness}

\begin{defn}Let $(X,\mathcal{B},\mu,F_1^t)$ and $(Y,\mathcal{A},\nu,F_2^t)$ be two ergodic probability measure preserving flows. A {\bf\emph{joining}} is an $(F_1 \times F_2)^t$-preserved measure with marginals $\mu$ and $\nu$. The product measure $\mu \times \nu$ is called the {\bf\emph{trivial joining}}. If the trivial joining is the only joining between two systems we say they are {\bf\emph{disjoint}}. 
\end{defn}
If $(X,\mathcal{B},\mu,F_1^t)$ and $(Y,\mathcal{A},\nu,F_2^t)$ are isomorphic with $\phi$ being the isomorphism then $(Id \times \phi)_*\mu$ is a joining. Thus if two systems are disjoint then they are not isomorphic (unless they are the one point system).

Let $(X,d)$ be a compact metric space and $(X,\mu)$ be a Borel probability measure space. 

\begin{defn}
Let $F^t$ be a measurable flow on $(X,d)$ that preserves a probability measure $\mu$. We say that $\{t_j\}_{j\in \N} \subset \R$ is a {\bf\emph{$c$-partial rigidity sequence}} if for all $j$ there exists $A_j \subset X$ with $\mu(A_j)\geq c$ so that 
\[\underset{j \to \infty}{\lim}\int_{A_j}d(F^{t_j}x,x)d\mu=0.\] We say $\{t_j\}$ is a {\bf\emph{partial rigidity sequence}} if it is a $c$-partial rigidity sequence for some $c>0$. 
\end{defn}
\begin{defn}Let $F^t$ be a measurable flow on $(X,d)$ that preserves a probability measure $\mu$. We say that $\{t_j\}_{j\in \N} \subset \R$ is a {\bf\emph{mixing sequence}} for $F^t$ if for all $f,g\in L^2(\mu)$ we have 
\[\underset{j \to \infty}{\lim} \left\langle f \circ F^{t_j},g\right\rangle=\int f d\mu \int g d\mu.\]  If $L \subset \R$ has the property that any $\{t_j\}_{j\in \N} \subset L$ with $t_j\to \infty$ is a mixing sequence, then we call it a {\bf\emph{mixing set}}. 
\end{defn}
The following lemma is easier than results in the literature and we include a proof for the convenience of the reader. 
\begin{lm}\label{lem:disjoint crit}Let $F_1^t$ and $F_2^t$ be two $\mu$-measure preserving and ergodic flows. If there is a mixing sequence for $F_1^t$ which is a partial rigidity sequence for $F_2^t$ then $F_1^t$ and $F_2^t$ are disjoint. 
\end{lm}
The proof uses the following notion: A continuous linear map $P:L^2((X,\mu)) \to L^2((Y,\nu))$ is called a {\bf\emph{Markov operator}} with adjoint operator $P^*$ if it satisfies
 \begin{enumerate}
\item $P\geq 0$ and $P^*\geq 0$;
\item $P\one_X=\one_Y$ and $P^*\one_Y=\one_X$;
\item $PU_{F_1^s}=U_{F_2^s}P$ for $s\in \R$ and where $U_{F_j^s}$ is the unitary operator associated to $F_j^s$. 
\end{enumerate}
Let $\tau$ be a  joining of $(F_1^t,\mu)$ and $(F_2^t,\nu)$.  This defines a  Markov operator $P_{\tau}:L^2(\nu) \to L^2(\mu)$  by  
$\int_BP_\tau(\one_A)d\mu=\tau(B \times A)$. Formally, this Markov operator $P_\tau$ is the conditional expectation associated to the disintegration of $\tau$ over $\nu$. The set of Markov operators are in a one-to-one correspondence with joinings and given a joining $\tau$ we will refer to the corresponding Markov operator $P_\tau$ as the {\bf\emph{associated Markov operator}}. This identification respects the convex structure of preserved measures and so  the extreme points come from ergodic joinings. For more on Markov operators, see \cite{Glas}.

\begin{proof}[Proof of \cref{lem:disjoint crit}] Let $\tau$ be a joining of $F_1^t$ and $F_2^t$ and $P$ be the associated Markov operator.  Choose $c>0$ and $\{t_j\}_{j\in \N}$ so that  $\{t_j\}$ is simultaneously a mixing sequence for $F_1^t$ and a $c$-partial rigidity sequence for $F_2^t$. 

Since $\{t_j\}$ is a mixing sequence for $F_1^t$, we have $PF_1^{t_j}=F_2^{t_j}P$ converge in the weak operator topology to the integral because the same is true for the unitary operators associated with $F_1^{t_j}$. Recall the weak operator topology is compact on operators with uniformly bounded operator norm. Thus, there is a subsequence of $t_{j_k}$ so that the sequence of unitary operators associated with $F_2^{t_{j_k}}$ converges in the weak operator topology to an operator $V$. Because $t_j$ is a $c$-partial rigidity sequence for $F_2^t$, $V$ necessarily has the form $c\mathrm{Id}+V'$ where $V'$ is a positive operator. Set $\int$ to be the bounded linear operator on $L^2$ given by integrating, so $(c\mathrm{Id}+V')P=\int$. Since $F_1^t$  has a mixing sequence, implying $F_1^t$ is weakly mixing, we have $(F_1 \times F_2)^t$ is ergodic. Thus $\int$ (the Markov operator associated to the product measure) is an extreme point in the set of Markov operators. The equality $(c\mathrm{Id}+V')P=\int$ combined with $\int$ being an extreme point implies $c\mathrm{Id}P=c\int$ (and $V'P=(1-c)\int$). Thus $P$ is in fact $\int$, implying the trivial joining is the only joining.
\end{proof}

The following disjointness criterion is motivated by the approach of \cite{ChaThesis} and is more or less easier than the criterion in \cite{CH}. Let $\{F_\theta^t\}_{\theta\in S^1}$ be the 1-parameter family of $\mu$-measure preserving flows on $X$ coming from the straight line flow in each direction on $X$. Let $\lambda$ be the Lebesgue measure on $S^1$. 
\begin{prop}\label{prop:criterion} If 
\begin{itemize}
\item $F_\theta^t$ is weakly mixing for $\lambda$-almost every $\theta$, and
\item for every $L \subset \R$, which is a union of intervals of length at least 1 and has density 1, 
we have that $F_\phi^t$ has a partial rigidity sequence in $L$ for
$\lambda$-almost every $\phi$,
\end{itemize} then for $\lambda^2$-almost every $(\theta,\phi)\in (S^1)^2$ we have that $F_\theta^t$ and  $F_\phi^t$ are disjoint.
\end{prop}

\begin{lm}\label{lem:mix seq}Let $F^t$ be a $\mu$-weakly mixing flow, and let $r<\infty$. We can find $L \subset \R$  which is a union of intervals of length at least $r$, has density 1, and any sequence in $L$ going to infinity is a mixing sequence for $F^t$. 
\end{lm}
The fact that there exists such an $L$ of density 1 is a well known equivalence of weak mixing. The fact that $L$ can be chosen as a union of intervals of length at least 1 is because the unitary operators for $F^t$ where $t\in [0,r]$ is a compact family in the strong operator topology. So, if $L'$ is a mixing set, its $r$-neighborhood is too.
\begin{proof}[Proof of Proposition \ref{prop:criterion}] It suffices to consider when $F_{\theta}^t$ is weakly mixing. By Lemma~\ref{lem:mix seq} choose a set $L$ satisfying the assumptions of Lemma~\ref{lem:mix seq} for $r=1$. By assumption there is a full measure set $D$ of directions that have a partial rigidity sequence in $L$. By Lemma \ref{lem:disjoint crit}, for all $\phi \in D$, $F_\theta^t$ and $F_\phi^t$ are disjoint, as desired. \end{proof}

\subsection{Application to non-arithmetic Veech surfaces}

We now set about adapting the criterion of Proposition~\ref{prop:criterion} to certain translation surfaces. The arguments have more or less been carried out in \cite{ChaHomog}, \cite{MTW} and \cite{CFKU}. For the sake of novelty, we present a slightly different, weaker (in particular avoiding \emph{ubiquity}), argument that is sufficient for our purposes. 
\begin{thm} \label{thm:disjoint trans} Let $M$ be a translation surface and $\mu$ denote the measure coming from the area. Let $F_\theta^t$ denote the flow in direction $\theta$ on $M$. 
\begin{itemize}
\item Assume the flow in almost every direction is weakly mixing. 
\item Suppose for every $\epsilon>0$ there exists $\eta>0$ so that all but at most an $\epsilon$-proportion of holonomies of cylinders do not have a $\eta$-friend. 
\end{itemize}
Then the family $\{F_\theta^t\}_{\theta\in S^1}$ with the measure $\mu$ satisfies the assumption of Proposition \ref{prop:criterion}. 
\end{thm}
\begin{coro}\label{cor:redux} Let $M$ be a non-arithmetic Veech surface. Then for almost every $\theta, \psi \in S^1$ we have $F_\theta$ and $F_\psi$ are disjoint. In particular, the flows are not isomorphic. 
Moreover, for almost every $\theta, \psi \in S^1$ the product flows are uniquely ergodic.
\end{coro}
\begin{proof}[Proof of Corollary~\ref{cor:redux} assuming Theorem \ref{thm:disjoint trans}] 
First, the flow in almost every direction on any non-arithmetic Veech surface is weak mixing \cite{AD}. \cref{thm:epsM} implies that every non-arithmetic Veech surface has the property that for every $\epsilon>0$ there exists $\eta>0$ so that for all but at most an $\epsilon$ proportion of holonomies of cylinders do not have a $\eta$-friend. These two results are the assumptions of Theorem~\ref{thm:disjoint trans} and so by Proposition~\ref{prop:criterion} the flows in almost every pair of directions on M are disjoint.

We now establish the typical unique ergodicity of the product.  Observe that the product of two flows is uniquely ergodic if and only if the flows are uniquely ergodic and disjoint. By \cite{KMS}, for every translation surface the flow in almost every direction is uniquely ergodic. Thus, for almost every $\theta,\psi$ we have $F_\theta \times F_\psi$ is uniquely ergodic. 
\end{proof}

We prove \cref{thm:disjoint trans} in the rest of this section. The key lemma is

\begin{lm}\label{lem:basic geom} 
%Assume for every $\epsilon>0$ there exists $\eta>0$ so that all but at most an $\epsilon$-proportion of holonomies of cylinders do not have a $\eta$-friend. Then for every $\delta>0$ there is $r<\infty$ so that if $L_{r}\subset \R$ is a union of intervals of length at least $r$ with density 1, then a density $1-\delta$ set of lengths of holonomies of cylinders is contained in $L_{r}$.

Let $M$ be a Veech surface. If $L\subseteq \R$ is a union of intervals of length at least $1$ with density $1$, then a density 1 set of holonomies of saddle connections have lengths contained in $L$. That is, we have
\[\lim_{R\to\infty} \frac{|\{(\theta, T)\in S_M: T\in L \cap [0,R)\}|}{|\{(\theta, T)\in S_M: T\in [0,R)\}|} = 1.\]
\end{lm}

\begin{proof}[Proof of Lemma~\ref{lem:basic geom}]
Denote an element in $S_M$ by its polar decomposition $(\theta, T)$. We want to find a bound for
\[\limsup_{R\to\infty}\frac{|\{(\theta, T)\in S_M: T\in L^c \cap [0,R)\}|}{|\{(\theta, T)\in S_M: T\in [0,R)\}|}.\]
Let $\delta>0$. We split the set in the numerator into two cases, first those elements with $\eta$-friends, and then those without $\eta$-friends. Having fixed $\delta$, choose $\eta$ as given by \cref{thm:epsM}. Then
\begin{align*}&\limsup_{R\to\infty}\frac{\abs{\{(\theta, T)\in S_M: T\in L^c \cap [0,R)\text{ and } (\theta,T)\text{ has an }\eta\text{-friend}\}}}{\abs{\{(\theta, T)\in S_M: T\in [0,R)\}}} < \delta.
% \\
%&\leq\limsup_{R\to\infty} \frac{\abs{\{(\theta, T)\in S_M: T\in [0,R)\text{ and } (\theta,T)\text{ has an }\eta\text{-friend}\}}}{\abs{\{(\theta, T)\in S_M: T\in [0,R)\}}} < \delta.
\end{align*}
We now consider the case of holonomy vectors with no $\eta$-friends, i.e., 
\begin{equation}\label{eq:A1}
\limsup_{R\to\infty}\frac{|\{(\theta, T)\in S_M: T\in L^c \cap [0,R)\text{ and } (\theta,T) \text{ has no }\eta\text{-friend}\}|}{|\{(\theta, T)\in S_M: T\in [0,R)\}|}.
\end{equation}
By quadratic growth, the denominator of \eqref{eq:A1} is bounded below by $\rho R^2$ for some $\rho>0$ and $R>0$. By the assumptions on $L$ the complement $L^c \cap [0, R)$ is a finite set of $q$ points, where $q < R$. We thicken $L^c$ by placing an interval $I_i$ of length $|I_i|=\ell$ at each point. Making $\ell$ small enough we define $L^c_\ell \cap [0,R)$ to be the disjoint union of the intervals: $L^c_\ell \cap [0,R)= \bigsqcup_{i=1}^q I_i$.

So by comparing areas, for any interval $I_i$ we have
\begin{equation*}
	|\{(\theta, T)\in S_M: T\in I_i \text{ and } (\theta,T)\text{has no }\eta\text{-friend}\}| \leq \frac{\pi(\ell+\eta)^2}{\pi \eta^2} = \frac{(\ell+\eta)^2}{\eta^2} .
\end{equation*}
Summing over the $q<R$ intervals, we find
\begin{equation}\label{eq:noetafriend}
|\{(\theta, T)\in S_M: T\in L_\ell^c \cap [0,R)\text{ and } (\theta,T)\text{has no }\eta\text{-friend}\}|\, \leq\ R \frac{(\ell+\eta)^2}{\eta^2}.
\end{equation}
Since $\ell$ can be taken to be non-increasing in $R$ and $\rho,\eta$ are independent of $R$, \eqref{eq:A1} yields
\[\limsup_{R\to\infty}\frac{\{(\theta, T)\in S_M: T\in L^c \cap [0,R)\text{ and } (\theta,T) \text{ has no }\eta\text{-friend}\}}{\{(\theta, T)\in S_M: T\in [0,R)\}} \leq \limsup_{R\to\infty} \frac{1}{R} \frac{(\ell+\eta)^2}{\rho\eta^2} = 0.\]
We conclude by letting $\delta\to0$.
\end{proof}

Let $M$ be an area 1 translation surface of genus $g$ and $s$ be the length of the shortest saddle connection on $M$. Set $\sigma=\frac 1 {2g-2}$. We now fix $L \subset \R$, a union of intervals of length at least 1 with density 1. Let
\[\tilde{S}_M(\sigma,N)=\{(\theta,T): \theta \text{ is the direction of a periodic cylinder of length }T < N \text{ and volume at least } \sigma\}.\] Let 
\[\tilde{S}'_M(\sigma, N)=\left\{(\theta,T)\in \tilde{S}_M({\sigma},N):T \in L\right\}.\]

\begin{prop}\label{prop:suff} It suffices to show that there exists $\hat{c}>0$ so that for any interval $J$, and for any $N$ large enough (depending on $J$) we have 
\begin{equation}\label{eqn:key}\lambda\left(\bigcup_{(\theta,T)\in \tilde{S}'_M({\sigma},N)}B\left(\theta,\frac 1{TN}\right)\cap J\right)>\hat{c}\lambda(J).
\end{equation} 
\end{prop}
\begin{proof}[Proof of Proposition~\ref{prop:suff}] We first need a lemma connecting being close to the directions of cylinders which have area at least $c>0$ to  partial rigidity sequences.
\begin{lm}\label{lem:rigid} If $v_i$ are the holonomies of the circumferences of cylinders on $Q$ with area at least $c$, $\theta_i$ are their directions and 
\[\underset{i \to \infty}{\lim}\, \abs{v_i}^2 \abs{\theta_i-\theta}=0\] then $|v_i|$ is a $c$-partial rigidity sequence for $F_{r_\theta Q}$.  
\end{lm}
The proof of Lemma~\ref{lem:rigid} follows directly from the proof of \cite[Lemma 12]{CH}. 

We now show that with Lemma~\ref{lem:rigid}, Proposition~\ref{prop:suff} follows from the Lebesgue density theorem because \eqref{eqn:key} for all intervals implies that for any $c'<\hat{c}$, for all large enough $N$ we have 
\begin{equation}
\label{eqn:key2}\lambda\left(\bigcup_{(\theta,T)\in \tilde{S}'_M({\sigma},N)}B\left(\theta,\frac \epsilon{TN}\right)\cap J\right)>\epsilon c'\lambda(J).
\end{equation}
To see how \eqref{eqn:key2} follows from \eqref{eqn:key},  if $B(\theta, \frac 1 {TN})\subset J$ then $\lambda( B(\theta, \frac \epsilon {TN}) \cap J)=\epsilon \lambda(B(\theta,\frac 1 {TN})\cap J)$. Now  
\[\lambda\left(J\cap \bigcup_{(\theta,T)\in \tilde{S}_M(\sigma,N):B(\theta,\frac 1{TN})\not \subset J}B\left(\theta,\frac 1 {TN}\right)\right)\leq |\partial J|\frac 1 {sN}=\frac 2{sN}\] which clearly goes to zero as $N$ goes to infinity. 

Now \eqref{eqn:key2} (and the fact that $N\geq T$) establishes that for any $\epsilon>0$, the complement of \[\bigcup_{(\theta,T)\in \tilde{S}'_M({\sigma},N)}B\left(\theta,\frac \epsilon{T^2}\right)\] has no Lebesgue density points and so $\cup_{(\theta,T)\in \tilde{S}'_M({\sigma},N)}B(\theta,\frac \epsilon{T^2})$ has full measure. So for almost every $\phi \in S^1$ there exists $(\theta_j,T_j)\in \tilde{S}'_M(\sigma, \infty)$ so that 
\[\underset{j \to \infty}{\lim}T_j^2|\theta_j-\phi|=0.\] If $\phi$ is not in the countable set of direction of elements of $\tilde{S}'_M({\sigma},\infty)$ then the $T_j$ necessarily tend to infinity. By Lemma \ref{lem:rigid}, the $T_j$ are a $\sigma$-partial rigidity sequence for $F_\phi^t$ and we have established the theorem. Thus we have shown that \eqref{eqn:key} implies Theorem \ref{thm:disjoint trans}. 
\end{proof}

To prove \eqref{eqn:key}, we use a result of Vorobets.

\begin{thm}\label{thm:vor} (Vorobets, \cite[Page 16]{Vor}) Let $m$ be the sum of the multiplicities of singularities. Let $c=2^{2^{4m}}$ 
For large enough $N$ we have 
\begin{equation}\label{eqn:keyVor}
\bigcup_{(\theta,T)\in \tilde{S}_M({\sigma},N)}B\left(\theta,\frac {c^2}{TN}\right)=S^1.
\end{equation}
\end{thm}
\begin{proof}[Derivation of Theorem \ref{thm:vor} from Vorobets] Let $(\theta,T) \in S^1 \times \R^+$. For $\alpha>0$ let $A_{(\theta,T)}(\alpha)$ be the set of directions $\phi$ so that  
\[T|\sin(\theta-\phi)|\leq \lambda^{-1} c.\] Vorobets shows that for any $N\geq c$ and $\lambda=\frac N c$ we have 
\[S^1=\bigcup_{(\theta,T)\in \tilde{S}_M(\sigma,N)}A_{(\theta,T)}(\lambda).\]
If $N$ is big enough, (depending on the shortest saddle connection in $M$)  ${T\sin(\theta-\phi)\leq (\frac Nc)^{-1} c}$ implies $|\sin(\theta-\phi)|\geq \frac 1 2 |\theta-\phi|$. 
So for large enough $N$ we have that for every $\phi$ there exists $(\theta,T) \in \tilde{S}_M(\sigma,N)$ so that $T\frac 1 2 |\theta-\phi|\leq T|\sin(\theta-\phi)|\leq (\frac Nc)^{-1} c$ or that $|\phi-\theta|\leq \frac {c^2}{TN}$. 
\end{proof}
\begin{rmk}For the reader's convenience, we briefly connect our notation to the notation in \cite{Vor}. There $T_0$ plays the role of $c$.  Because we are assuming our translation surface has area 1 the $\sqrt{S}$ term (which denotes the area of the translation surface) does not appear. Also note that Vorobets allows for a smaller $c$ when $m=1,2$. In \cite{Vor}, $T$ plays the role of $R$. 
\end{rmk}

\begin{proof}[Proof of Theorem \ref{thm:disjoint trans}] By Proposition \ref{prop:suff} it suffices to show that \eqref{eqn:keyVor} implies \eqref{eqn:key}. Let $J$ be an interval. Observe that  
\begin{multline}\label{eq:wrap}\lambda\left(\bigcup_{(\theta,T)\in \tilde{S}'_M({\sigma},N)}B\left(\theta,\frac 1{TN}\right)\cap J\right)\geq \\
\frac{1}{c^2}\lambda\left(\bigcup_{(\theta,T)\in \tilde{S}_M({\sigma},N)}B\left(\theta,\frac {c^2}{TN}\right)\cap J\right)-\sum_{(\theta,T) \in \tilde{S}_M(\sigma,N)\setminus \tilde{S}'_M(\sigma,N)} 2\frac {c^2}{TN}.
\end{multline}
By Lemma \ref{lem:basic geom} we may assume $|\tilde{S}_M(\sigma,N)\setminus \tilde{S}'_M(\sigma,N)|<\delta N^2$ for all $N$ large enough. Thus for large enough $N$ , by Abel summation,
\[\sum_{(\theta,T) \in \tilde{S}_M(\sigma,N)\setminus \tilde{S}'_M(\sigma,N)} 2\frac {c^2}{TN} < 4c^2\delta.\]
 So for any $\delta>0$ we can bound the right hand side of  \eqref{eq:wrap} by $\frac 1 {c^2}\lambda(J)-C\delta$. This establishes  \eqref{eqn:key}.
\end{proof}

\begin{rmk} In the previous referenced uses of similar disjointness arguments, \cite{ChaThesis}, \cite{CH} one does not assume that the systems are weakly mixing. The author of the appendix considers one of the values of this appendix to be presenting a disjointness criterion of this flavor that is simplified by the presence of density one mixing sequences.  However, similar to those earlier arguments, one can remove the condition that the flow in almost every direction on $M$ is weakly mixing by using \cite[Proposition 2]{CH} in place of Lemma \ref{lem:disjoint crit}. Lemma \ref{lem:basic geom} provides the necessary input for this criterion as well. Compare with \cite[Corollary 5]{ChaThesis} which uses \cite[Corollary 3]{ChaThesis}. 
By a similar argument to the one above, one can show that any Veech surface (including arithmetic Veech surfaces) has the property that the flow in almost every pair of directions is disjoint. However, even without the main result of the present paper, one can use similar spectral methods to prove that for any branched translation cover  of a torus (a class of surfaces that includes all arithmetic Veech surfaces, which are branched translation covers of tori, branched over one point) the flow in almost every pair of directions is disjoint and so we omit it. 
\end{rmk}

%------
% Insert acknowledgments and information
% regarding funding at the end of the last
% section, i.e., right before the bibliography.
%------

\begin{ack}
We would like to thank MSRI (now SLMath) where this project was started in the Fall of 2019 during the program on Holomoprhic Differentials in Mathematics and Physics.  The authors would like to thank Jayadev Athreya, Max Goering, Jiyoung Han, and Pedram Safaee for useful discussions, and the anonymous referees for their comments. 
\end{ack}

\begin{funding}
C.~B.~is supported by the Swiss National Science Foundation Grant No.~201557, and would like to thank the Hausdorff Institute for Mathematics in Bonn, where part of this work was completed in the Summer of 2021. S.~F.~was partially supported by the Deutsche Forschungsgemeinschaft (DFG)--Projektnummer 44546644. J.~C.~ is supported by NSF grants DMS-2055354 and DMS-452762, the Sloan foundation, Poincar{\'e} chair, and Warnock chair.
\end{funding}

%------
% Insert the bibliography.
%------

%\begin{thebibliography}{99}

%------ Example for a paper in journal:
% \bibitem{article1}
% Petrunin, A.: Parallel transportation for Alexandrov space with curvature bounded below.
% Geom. Funct. Anal. \textbf{8}, 123--148 (1998) \Zbl{0903.53045} \MR{1601854}

%------ Example for a book:
% \bibitem{book1}
% Ziemer, W.~P.: Weakly differentiable functions. Grad. Texts in Math. 120,
% Springer, New York (1989) \Zbl{0692.46022} \MR{1014685}

%------ Example for a paper in a book:
% \bibitem{incollection1}
% Milne, J.~S.: Introduction to Shimura varieties. In: Harmonic Analysis, the
% Trace Formula, and Shimura Varieties (M.~W. Marcellin, E.~Giorgi, eds.),
% pp. 265--378, Clay Math. Proc. 4,
% American Mathematical Society, Providence, RI, 2005
% \Zbl{1148.14011} \MR{2192012}

%------ Example for a preprint on arXiv:
% \bibitem{preprint1}
% Nguyen, D.~V., Chilappagari, S.~K., Marcellin, M.~W., Vasic, B.:
% LDPC codes from latin squares free of small trapping sets.
% \arxiv{1008.4177} (2010)

%------ Example for a report:
% \bibitem{report1}
% Schöberl, J.: Commuting quasi-interpolation operators.
% Technical report isc-01-10-math, Texas A\&M University,
% \url{www.isc.tamu.edu/publications-reports/tr/0110.pdf} (2001)

%------ Example for a thesis:
% \bibitem{thesis1}
% Giorgi, E.: The geometric universe.
% Ph.D. thesis, University of Maryland, College Park (2002)

%\end{thebibliography}
\bibliographystyle{emsjems}
\bibliography{Sources}
\end{document}